\documentclass[pre,aps,superscriptaddress,nofootinbib,notitlepage,10pt]{revtex4-2}

\usepackage[margin=1in]{geometry}

\usepackage[
  colorlinks,
  linkcolor = blue,
  citecolor = blue,
  urlcolor = blue]{hyperref}

\usepackage{amsmath}
\usepackage{amsfonts}
\usepackage{amsthm}
\usepackage{amssymb}
\usepackage{cleveref}
\usepackage[normalem]{ulem}
\usepackage[usenames,dvipsnames]{xcolor}
\usepackage{stackrel}
\usepackage{centernot}
\usepackage{graphicx}
\usepackage{enumitem}
\usepackage{comment}

\newlist{thmlist}{enumerate}{1}
\setlist[thmlist]{label=(\roman{thmlisti}), ref=(\roman{thmlisti}),noitemsep}
\makeatletter
\renewcommand{\p@thmlisti}{\perh@ps{\thetheorem}}
\protected\def\perh@ps#1#2{\textup{#1#2}}
\newcommand{\itemrefperh@ps}[2]{\textup{#2}}
\newcommand{\itemref}[1]{\begingroup\let\perh@ps\itemrefperh@ps\ref{#1}\endgroup}
\makeatother

\newtheorem{theorem}{Theorem}
\newtheorem{lemma}[theorem]{Lemma}
\newtheorem{corollary}[theorem]{Corollary}

\crefname{conjecture}{conjecture}{conjectures}
\Crefname{conjecture}{Conjecture}{Conjectures}
\newtheorem{definition}[theorem]{Definition}

\newtheorem{example}[theorem]{Example}

\Crefname{observation}{Observation}{Observations}

\crefname{game}{game}{games}
\Crefname{game}{Game}{Games}

\crefname{secinapp}{Appendix}{Appendices}
\Crefname{secinapp}{Appendix}{Appendices}

\newcommand{\braopket}[3]{\ensuremath{\left\langle{#1}\middle|{#2}\middle|{#3}\right\rangle}}
\newcommand{\bra}[1]{\ensuremath{\langle{#1}|}}
\newcommand{\ket}[1]{\ensuremath{|{#1}\rangle}}
\newcommand{\braket}[2]{\ensuremath{\langle{#1}|{#2}\rangle}}

\newcommand{\abs}[1]{\ensuremath{{\lvert#1\rvert}}}
\newcommand{\ip}[1]{\langle{#1}\rangle}
\newcommand{\hilb}[1]{\ensuremath{\mathcal{#1}}}

\newcommand{\hilbdual}[1]{{{#1}^\dag}}
\newcommand{\rank}[0]{\ensuremath{\textnormal{rank}}}

\newcommand{\linspan}[0]{\ensuremath{\textnormal{span}}}
\newcommand{\diag}[0]{\ensuremath{\textnormal{diag}}}
\newcommand{\Tr}[0]{\ensuremath{\textnormal{Tr}}}
\newcommand{\linop}[1]{\ensuremath{\mathcal{L}(#1)}}
\newcommand{\psd}[1]{\ensuremath{\textnormal{Pos}(#1)}}
\newcommand{\herm}[1]{\ensuremath{\textnormal{Herm}(#1)}}
\newcommand{\proj}[1]{{\ket{#1}\bra{#1}}}

\newcommand{\conj}[1]{\overline{#1}}

\newcommand{\hA}[0]{\hilb{A}}
\newcommand{\hAi}[0]{{\hilb{A}_i}}
\newcommand{\hAj}[0]{{\hilb{A}_j}}
\newcommand{\hB}[0]{\hilb{B}}
\newcommand{\hBi}[0]{{\hilb{B}_i}}
\newcommand{\hBj}[0]{{\hilb{B}_j}}

\newcommand{\hY}[0]{\hilb{Y}}
\newcommand{\hYi}[0]{{\hilb{Y}_i}}
\newcommand{\hYj}[0]{{\hilb{Y}_j}}
\newcommand{\hZ}[0]{\hilb{Z}}
\newcommand{\hZi}[0]{{\hilb{Z}_i}}
\newcommand{\hZj}[0]{{\hilb{Z}_j}}
\newcommand{\TrA}[0]{\Tr_\hA}
\newcommand{\TrB}[0]{\Tr_\hB}

\newcommand{\TrAi}[0]{\Tr_\hAi}
\newcommand{\TrBi}[0]{\Tr_\hBi}
\newcommand{\TrYi}[0]{\Tr_\hYi}

\newcommand{\ot}[0]{\otimes}
\newcommand{\Gc}[0]{{\overline{G}}}

\newcommand{\Sc}[0]{{S^c}}
\newcommand{\SG}[0]{{S_G}}
\newcommand{\SGc}[0]{{S_\Gc}}
\newcommand{\lovasz}[0]{Lov{\'a}sz}

\newcommand{\opnorm}[1]{\ensuremath{\left\lVert#1\right\rVert}}
\newcommand{\trnorm}[1]{\ensuremath{\left\lVert#1\right\rVert_{\Tr}}}

\newcommand{\simG}[0]{\sim_G}
\newcommand{\simeqG}[0]{\simeq_G}

\newcommand{\dsw}[0]{\tilde{\vartheta}}
\newcommand{\lov}[0]{\vartheta}

\newcommand{\dswbody}[0]{\textnormal{TH}}
\newcommand{\dswbodyS}[0]{\textnormal{TH}^\sharp}
\newcommand{\thbody}[0]{\textnormal{TH}}
\newcommand{\diags}[0]{\mathcal{D}}
\newcommand{\vv}[0]{V}
\newcommand{\vw}[0]{W}
\newcommand{\vx}[0]{X}
\newcommand{\rw}[0]{\sqrt{W}}

\newcommand{\her}[0]{\ensuremath{\textnormal{her}}}
\newcommand{\conv}[0]{\ensuremath{\textnormal{conv}}}
\newcommand{\convcl}[0]{\ensuremath{\overline{\textnormal{conv}}}}
\newcommand{\corn}[1]{\mathcal{#1}}
\newcommand{\cornA}[0]{\corn{A}}
\newcommand{\cornB}[0]{\corn{B}}
\newcommand{\cornC}[0]{\corn{C}}

\newcommand{\spaceD}[0]{\mathcal{D}}
\newcommand{\vecmod}[2]{#1 / #2}
\newcommand{\twirlU}[0]{\Omega}
\newcommand{\blockscale}[0]{\Psi}
\newcommand{\commproj}[0]{\Delta}
\newcommand{\diagscale}[0]{D}
\newcommand{\dimAi}[0]{\dim(\hAi)}
\newcommand{\dimYi}[0]{\dim(\hYi)}

\newcommand{\relint}[0]{\ensuremath{\textnormal{relint}}}
\newcommand{\fp}[0]{\ensuremath{\textnormal{fp}}}
\newcommand{\vp}[0]{\ensuremath{\textnormal{vp}}}
\newcommand{\fg}[0]{\ensuremath{\textnormal{fg}}}
\newcommand{\abpsi}[0]{\diamond}

\begin{document}

\author{Dan Stahlke}
\title{Weighted theta functions for non-commutative graphs}

\date{\today}

\begin{abstract}
    Gr\"{o}tschel, Lov{\'a}sz, and Schrijver generalized the
    Lov{\'a}sz $\lov$ function by allowing a weight for each vertex.
    We provide a similar generalization of Duan, Severini, and Winter's $\tilde{\vartheta}$ on
    non-commutative graphs.
    While the classical theory involves a weight vector assigning a non-negative weight to each
    vertex, the non-commutative theory uses a positive semidefinite weight matrix.
    The classical theory is recovered in the case of diagonal weight matrices.

    Most of Gr\"{o}tschel, Lov{\'a}sz, and Schrijver's results generalize to non-commutative graphs.
    In particular, we generalize the inequality
    $\vartheta(G, w) \vartheta(\Gc, x) \ge \langle w, x \rangle$ with some modification needed
    due to non-commutative graphs having a richer notion of complementation.
    Similar to the classical case, facets of the theta body correspond to cliques and
    if the theta body anti-blocker is finitely generated then it is equal to the non-commutative
    generalization of the clique polytope.

    We propose two definitions for non-commutative perfect graphs, equivalent for classical graphs
    but inequivalent for non-commutative graphs.
\end{abstract}

\maketitle

\section{Introduction}

\lovasz{}~\cite{lovasz79} introduced the $\lov$ function of a graph as an upper bound on the Shannon capacity
-- the independence number regularized under the strong graph product.
The $\lov$ quantity is an upper bound on independence number, a lower bound on fractional
chromatic number, and is multiplicative under the strong and the disjunctive graph products.
It is a semidefinite program, hence efficiently computable both in theory and in practice.
It is monotone under graph homomorphisms~\cite{de2013optimization}; in fact its bound on
independence and chromatic number follow from this.

Further insight into $\lov$ is gained by allowing vertices to be
weighted~\cite{Grtschel1986,knuth94}.
Weights are basically equivalent to duplicating vertices~\cite{knuth94}
except that weights don't have to be whole numbers.
Aside from only being defined for non-negative weights, the weighted $\lov$ of a graph resembles a
norm on the weight vector: it scales linearly and is convex.
In that language, $\lov$ of the complement graph is the dual norm.
The set of weights $w$ for which $\lov(\Gc, w) \le 1$ is investigated
in~\cite{Grtschel1986}, where facets of this convex body are shown to correspond to clique constraints.
This set is polyhedral if and only if the graph is perfect.

\lovasz{}'s bound can be adapted to quantum channels via a suitable generalization of
graphs where an operator subspace takes the place of the adjacency matrix~\cite{dsw2013}.
These so called \emph{non-commutative graphs} have since drawn interest in connection with
quantum channels but also independently of any application.
Several classical graph definitions and results carry over to non-commutative graphs,
including
homomorphisms~\cite{Stahlke2016QuantumZS,weaver_quantum_relations,ortiz2016quantum,brannan2020quantumtoclassical},
chromatic numbers~\cite{Stahlke2016QuantumZS,KIM2019291,levene_paulsen_todorov_complexity},
Ramsey and Tur{\'a}n theorems~\cite{weaver_ramsey,weaver_qm_turan},
asymptotic spectrum~\cite{li2019quantum},
a Haemers bound~\cite{gribling2020haemers},
and connectivity~\cite{chavezdominguez2019connectivity}.
It can happen that there are multiple ways to generalize a particular
concept: \cite{btw2019}~presents two generalizations of $\lov$ distinct from the one
in~\cite{dsw2013} (though possibly the same as each other).

The present work investigates a weighted version of the $\dsw$ of~\cite{dsw2013}, generalizing most
of the results from~\cite{Grtschel1986}.
We note that~\cite{btw2019} defined a weighted version of their $\lov$ generalization,
and investigated the corresponding theta body.
It is not known whether it supports the sort of duality relations we find in this paper.

We will cover basic notation in \cref{sec:notation},
introduce our weighted $\dsw$ in \cref{sec:weighted_dsw},
prove a duality relation in \cref{sec:holder},
extend this to $S_0$-graphs in \cref{sec:block_graphs}
(with the core proof deferred to \cref{sec:main_proof}),
and in \cref{sec:facets} explore the geometry of the theta body for non-commutative graphs and its
relation to perfect graphs.
The presentation is self-contained, not requiring any background in quantum mechanics
or quantum channels.

\section{Notation and basic definitions}
\label{sec:notation}

Hilbert spaces will be denoted by the symbols $\hA, \hB, \hY, \hZ$.
These will always be finite dimensional.
Dual spaces are denoted $\hilbdual{\hA}$, etc.
Linear operators on $\hA$ are denoted $\linop{\hA}$; linear maps from
$\hA$ to $\hB$ by $\linop{\hA \to \hB}$.
The set of positive semidefinite operators on $\hA$ is denoted $\psd{\hA}$,
the set of Hermitian operators by $\herm{\hA}$.
For operators $X,Y \in \linop{\hA}$, $X \le Y$ means $Y - X \in \psd{\hA}$.
For spaces $\hA$ and $\hB \subseteq \hA$ we write the quotient space as
$\vecmod{\hA}{\hB} = \hA \cap \hB^\perp$.

We will use Dirac's bra-ket notation where $\ket{x}$ is a vector,
$\bra{x}$ is its dual, $\braket{x}{y}$ is an inner product, and
$\braopket{x}{M}{y}$ is an inner product between $\bra{x}$ and
$M \ket{y}$ with $M$ an operator.
The adjoint of an operator is written $M^\dag$.
Basis vectors are denoted $\ket{i}$ or $\ket{j}$ with $i, j \in \{1, \dots, n \}$.

For Hilbert spaces $\hA$ and $\hB$ of equal dimension,
we choose an isomorphism $\hilbdual{\hA} \to \hB$, denoted $\ket{\Phi}$.
There is a canonical choice once an orthonormal basis has been chosen:
\begin{align}
    \ket{\Phi} &= \sum_i \ket{i}_\hA \ot \ket{i}_\hB.
\end{align}
For an operator $M \in \linop{\hA}$ we define its vectorization
$\ket{M} = (M \ot I) \ket{\Phi} \in \hA \ot \hB$.
We also use this isomorphism to define the transpose, taking it to move
an operator from $\linop{\hA}$ to $\linop{\hB}$ so that
$(M \ot I) \ket{\Phi} = (I \ot M^T) \ket{\Phi}$.
Complex conjugate is defined similarly: $\conj{M} = (M^\dag)^T$ and
$(M^\dag \ot I) \ket{\Phi} = (I \ot \conj{M}) \ket{\Phi}$.

We take the definition of a non-commutative graph from~\cite{dsw2013}.

\begin{definition}
    A \emph{non-commutative graph} is an operator subspace $S \subseteq \linop{\hA}$ satisfying
    $S = S^\dag$ and $I \in S$.
    \label{def:noncomm_graph}
\end{definition}

In the above definition and throughout we use the shorthand notation
$S^\dag = \linspan\{ x^\dag : x \in S \}$,
$S + S' = \linspan\{ x+y : x \in S, y \in S' \}$,
$\mathbb{C} I = \linspan\{ x I : x \in \mathbb{C} \}$, etc.

Non-commutative graphs are analogous to adjacency matrices,
but taken as subspaces rather than 0-1 matrices.
Indeed, given any classical graph $G$ we can define a corresponding
non-commutative graph.

\begin{definition}
    For a graph $G$ define $\SG = \linspan\{ \ket{i}\bra{j} : i \simeqG j \}$.
    \label{def:classical_graph}
\end{definition}

Note that~\cite{Stahlke2016QuantumZS} breaks from the above definition,
taking non-commutative graphs to be trace free, $I \perp S$.
That works better for graph homomorphisms, which require vertices to not be self adjacent.
But for the present work the convention of \cref{def:noncomm_graph,def:classical_graph}
is more appropriate.

We take the following definition from~\cite{btw2019}.
This is an extension to operators of the convex corners defined in~\cite{Fulkerson1971,Csiszr1990},
which in turn are recovered by restricting to diagonal operators.
\begin{definition}
    Let $\hA$ be a Hilbert space.
    A convex corner in $\linop{\hA}$ is a non-empty closed convex subset
    $\cornC \subseteq \psd{\hA}$ such that
    \begin{align}
        A \in \cornC \textrm{ and } 0 \le B \le A \implies B \in \cornC.
        \label{eq:her_cond}
    \end{align}
    \label{def:convex_corner}
\end{definition}
Condition~\eqref{eq:her_cond} is called \emph{herditarity}.
For a general subset $\cornC \subseteq \psd{\hA}$ we denote by $\her(\cornC)$ the
\emph{hereditary closure}, the smallest hereditary set containing $\cornC$:
\begin{align}
    \her(\cornC) = \{ A \in \psd{\hA} : \exists B \in \cornC \textrm{ s.t. } A \le B \}.
    \label{eq:her_clos}
\end{align}
The smallest convex corner containing a given subset $\cornC \subseteq \psd{\hA}$ of positive
semidefinite operators is obtained by taking the hereditary closure of the closure of the convex
hull,
\begin{align}
    \her( \convcl( \cornC )).
\end{align}
We then say this convex corner is \emph{generated} by $\cornC$.  We say a convex corner is
\emph{finitely generated} if it can be generated by a finite set.

Also from~\cite{btw2019}, and again an extension of a concept from~\cite{Fulkerson1971,Csiszr1990}, we take the
definition of anti-blockers.
\begin{definition}
    Let $\hA$ be a Hilbert space.
    The anti-blocker of $\cornC \subseteq \psd{\hA}$ is
    \begin{align}
        \cornC^\sharp = \{ B \in \psd{\hA} : \Tr(AB) \le 1 \textrm{ for all } A \in \cornC \}.
    \end{align}
    \label{def:antiblocker}
\end{definition}
The following theorem is from \cite[lemma 2.2.10 and theorem 2.3.12]{borelandthesis}.
\begin{lemma}
    Let $\hA$ be a Hilbert space.
    Convex corners satisfy the following basic facts.
    \begin{thmlist}
        \item If $\cornC \subseteq \psd{\hA}$ is non-empty then $\cornC^\sharp$ is a convex corner.
            \label[theorem]{thm:antiblocker_is_corner}
        \item If $\cornB, \cornC \subseteq \psd{\hA}$ and $\cornB \subseteq \cornC$ then
            $\cornC^\sharp \subseteq \cornB^\sharp$
            \label[theorem]{thm:antiblocker_inclusion}
        \item A non-empty set $\cornC \subseteq \psd{\hA}$ satisfies $\cornC =
            \cornC^{\sharp\sharp}$ if and only if $\cornC$ is a convex corner
            (the second anti-blocker theorem).
            \label[theorem]{thm:second_antiblocker}
    \end{thmlist}
\end{lemma}

Following~\cite{btw2019}, we will occasionally refer to \emph{diagonal convex corners}, analogous
to \cref{def:convex_corner} but restricted to diagonal matrices.
Note this is equivalent to the classical definition from~\cite{Fulkerson1971,Csiszr1990}, just with
the elements of the corner being positive semidefinite diagonal matrices rather than entrywise
non-negative vectors.
If $\cornC$ is a convex corner then $\cornC \cap \diags$, with $\diags$ being the subspace
of diagonal matrices, is a diagonal convex corner.
Again following~\cite{btw2019} we define the diagonal anti-blocker
\begin{align}
    \cornC^\flat &= \{ B \in \diags : B \ge 0, \Tr(AB) \le 1 \textrm{ for all } A \in \cornC \}
    \\ &= \cornC^\sharp \cap \diags.
\end{align}
This is equivalent to the classical notion of anti-blocker from~\cite{Fulkerson1971,Csiszr1990}.
We will only be referencing diagonal convex corners and diagonal anti-blockers when comparing
our theory to the previously established theory of classical (commutative) graphs.


\section{Weighted \texorpdfstring{$\dsw$}{theta}}
\label{sec:weighted_dsw}

The following definition of weighted $\lov$ for classical graphs is derived in
a straightforward way from~\cite[section 6]{knuth94}.  It is presented here for background
and to motivate a similar definition for non-commutative graphs.  We will not be using it
directly.

\begin{definition}
    Let $G$ be a graph with $n$ vertices and $\ket{w} \in \mathbb{R}^n$ be an entrywise positive vector.
    Let $\ket{r}$ be the entrywise square root of $\ket{w}$.  Define
    \begin{align}
        \lov(G, w) = \min\{ \lambda : Y \ge \ket{r} \bra{r},
            Y_{ii} = \lambda, Y_{ij} = 0 \textrm{ for } i \not\simeq j \}.
    \end{align}
    The unweighted $\lov$ is recovered by taking $\ket{w}$ to be the all ones vector.
    \label{def:classical_weighted}
\end{definition}

The non-weighted $\dsw$ has been defined for non-commutative graphs
by~\cite{dsw2013}, which presents the following equivalent definitions.

\begin{definition}
    Let $S$ be a non-commutative graph.
\begin{align}
    \dsw(S) &= \max\{\opnorm{T + I \ot I} : T \in S^\perp \ot \linop{\hB}, T + I \ot I \ge 0\}
    \\ &= \max\{ \braopket{\Phi}{T + I \ot \rho}{\Phi} :
        T \in S^\perp \ot \linop{\hB}, T + I \ot \rho \ge 0, \rho \ge 0, \Tr \rho = 1 \}
    \\ &= \min\{ \opnorm{ \TrA Y } : Y \in S \ot \linop{\hB}, Y \ge \ket{\Phi}\bra{\Phi} \}
    \\ &= \min\{ \lambda : Y \in S \ot \linop{\hB}, \TrA Y \le \lambda I, Y \ge \ket{\Phi}\bra{\Phi} \}
    \label{eq:unweighted_dsw_min_Y}
\end{align}
    \label{def:unweighted_dsw}
\end{definition}

If $G$ is a graph and $\SG = \linspan\{ \ket{i}\bra{j} : i \simeqG j \}$ then
$\dsw(\SG) = \lov(G)$.

It is worth mentioning that for classical graphs there are many alternate
forms for $\lov$~\cite{Grtschel1986,knuth94}.
While these are all equivalent for classical graphs, they are in general different
when generalized to non-commutative graphs.
Two quantities different from~\cref{def:unweighted_dsw} are studied in~\cite{btw2019}.

We now construct a weighted version of $\dsw(S)$.
Observing the similarity between \eqref{eq:unweighted_dsw_min_Y} and
\cref{def:classical_weighted}, it seems reasonable that the weights should be absorbed
into $\ket{\Phi}$.
The weights will be a positive semidefinite operator rather than an entrywise positive vector,
which is common for quantum generalizations of classical concepts  (cf.\ density operators vs.\
probability distributions).

\begin{definition}
    Let $S$ be a non-commutative graph and $W \in \psd{A}$.
    Let
    $\ket{\vw} = (      W  \ot I) \ket{\Phi} \in \hA \ot \hB$ and
    $\ket{\rw} = (\sqrt{W} \ot I) \ket{\Phi}$.
    Note that $\TrB\left( \proj{\rw} \right) = W$ and
    $\TrA\left( \proj{\rw} \right) = W^T$.
    Define
    \begin{align}
        \dsw(S, W) &= \min\left\{ \lambda :
            Y \in S \ot \linop{\hB},
            \TrA Y \le \lambda I,
            Y \ge \proj{\rw} \right\}.
        \label{eq:weighted_dsw_min_Y}
    \end{align}
    \label{def:weighted_dsw_min_Y}
\end{definition}

Note that for $W = I$,~\eqref{eq:weighted_dsw_min_Y} reduces
to~\eqref{eq:unweighted_dsw_min_Y}.
And for a classical graph with diagonal weight matrix it reduces to
\cref{def:classical_weighted}:

\begin{theorem}
    For any classical graph $G$ and weight vector $\ket{w}$,
    \begin{align}
        \dsw(\SG, \diag(w)) = \lov(G, w)
    \end{align}
    where $\SG = \linspan\{ \ket{i}\bra{j} : i \simeqG j \}$
    and $\diag(w)$ is the diagonal matrix with $\diag(w)_{ii} = w_i$.
    \label{thm:match_classical}
\end{theorem}
\begin{proof}
    Straightforward generalization of the proof of~\cite[corollary 12]{dsw2013}.
\end{proof}

\Cref{thm:match_classical} invites a question: for classical graphs, does our $\dsw$,
by allowing non-diagonal weight matrices, provide additional information about the graph?
In \cref{sec:block_graphs} we will find the answer is no.
For classical graphs, $\dsw$ with non-diagonal weights is a function of $\dsw$ with
diagonal weights (\cref{thm:dswSW}).

Being a semidefinite program,~\eqref{eq:weighted_dsw_min_Y} is efficiently computable.
It is more computationally expensive than the classical $\lov$ because $Y$
is of size $n^2 \times n^2$ rather than $n \times n$ where $n = \dim(\hA) = \abs{G}$.
Experiments with the SCS solver~\cite{scs2016,convexjl} show $n=9$ takes one minute (i7-6820HQ CPU,
circa 2015), with runtime scaling at about $O(n^6)$.
Much insight can be gained through numerical experiments even with $n=3,4$.
And, crucially, semidefinite programs have dual formulations.

\begin{theorem}
    The dual of the semidefinite program~\eqref{eq:weighted_dsw_min_Y} is
    \begin{align}
        \dsw(S, W) &= \max\Big\{
        \braopket{\rw}{T + I \ot \rho}{\rw} :
        T \in S^\perp \ot \linop{\hB},
    \nonumber \\ &\qquad\qquad
        T + I \ot \rho \ge 0,
        \rho \ge 0, \Tr \rho = 1 \Big\}
        \label{eq:weighted_dsw_max_T}
    \end{align}
\end{theorem}
\begin{proof}
    Set $M = \proj{\rw}$ and rewrite \eqref{eq:weighted_dsw_min_Y},
    \begin{align}
        \dsw(S, W) &= \min\Big\{ \lambda :
            Y \in S \ot \linop{\hB},
            \TrA Y - \lambda I \le 0,
            M - Y \le 0 \Big\}.
    \end{align}
    The Lagrangian is
    \begin{align}
        L(\lambda, Y; R, \rho) &= \lambda
            + \ip{\rho, \TrA Y - \lambda I}
            + \ip{R, M - Y}
        \\ &= \ip{I \ot \rho - R, Y}
            + \lambda (1 - \Tr \rho)
            + \ip{R, M}.
    \end{align}
    The dual program is then
    \begin{align}
        d^* &= \max\Big\{ \ip{R, M} :
            I \ot \rho - R \in S^\perp \ot \linop{\hB},
            \Tr \rho = 1,
            R \ge 0, \rho \ge 0 \Big\}
    \end{align}
    Defining $T = R - I \ot \rho$ gives the right hand side of \eqref{eq:weighted_dsw_max_T}.
    The point $T=0, \rho = I_\hB / \dim(\hB)$ is in the relative interior of the
    feasible region so Slater's condition holds and $d^* = \dsw(S, W)$.
\end{proof}

\begin{theorem}
    Let $S \subseteq \linop{\hA}$ be a non-commutative graph and $W \in \psd{A}$.
    Let $n = \dim(\hA)$.
    We have the following alternate forms for $\dsw(S, W)$,
    where~\eqref{eq:weighted_dsw_min_YWinvT} requires $W$ to be non-singular.
    \begin{align}
        \dsw(S, W)
        &= \min\{ \lambda :
            Y \in S \ot \linop{\hB},
            \TrA Y = \lambda I,
            Y \ge \proj{\rw} \}
            \label{eq:weighted_dsw_min_Y_eq}
        \\ &= \min\{ \lambda :
            Y \in S \ot \linop{\hB},
            \TrA Y = \lambda W^T,
            Y \ge \ket{\vw}\bra{\vw} \}
            \label{eq:weighted_dsw_min_YWT_eq}
        \\ &= \min\{ \lambda :
            Y \in S \ot \linop{\hB},
            \TrA Y \le \lambda W^T,
            Y \ge \ket{\vw}\bra{\vw} \}
            \label{eq:weighted_dsw_min_YWT}
        \\ &= \min\{ \lambda :
            Y \in S \ot \linop{\hB},
            \TrA Y \le \lambda W^{-T},
            Y \ge \ket{\Phi}\bra{\Phi} \}
            \label{eq:weighted_dsw_min_YWinvT}
    \end{align}
    \begin{align}
        \dsw(S, W)
        &= \max\Big\{
            \opnorm{(\sqrt{W} \ot I)(T + I \ot I)(\sqrt{W} \ot I)} :
            T \in S^\perp \ot \linop{\hB},
            T + I \ot I \ge 0,
            \Big\}
            \label{eq:weighted_dsw_max_opnorm}
        \\ &= \max\Big\{
            \opnorm{\sqrt{T + I \ot I}(W \ot I)\sqrt{T + I \ot I}} :
            T \in S^\perp \ot \linop{\hB},
            T + I \ot I \ge 0,
            \Big\}
            \label{eq:weighted_dsw_max_opnorm2}
        \\ &= \max\Big\{
            n \opnorm{(\sqrt{W} \ot I) Y (\sqrt{W} \ot I)} :
            Y \in (S^\perp + \mathbb{C}I) \ot \linop{\hB},
            Y \ge 0,
            \TrA Y = I
            \Big\}
            \label{eq:weighted_dsw_max_Y}
        \\ &= \max\Big\{
            n \braopket{\rw}{Y}{\rw} :
            Y \in (S^\perp + \mathbb{C}I) \ot \linop{\hB},
            Y \ge 0, \Tr Y = 1 \Big\}
            \label{eq:weighted_dsw_max_Y_v2}
    \end{align}
    \label{thm:weighted_dsw_equiv_forms}
\end{theorem}
\begin{proof}
    Let $n = \dim(\hA)$.
    Let $Y$ be feasible for~\eqref{eq:weighted_dsw_min_Y} with value $\lambda$.
    Set $Y' = Y + n^{-1} I_\hA \ot (\lambda I_\hB - \Tr_\hA Y)$.
    Then $\Tr_\hA Y' = \lambda I$.
    Since $I_\hA \in S$ we have $Y' \in S \ot \linop{\hB}$.
    Since $\Tr_\hA Y \le \lambda I_\hB$, we have $Y' \ge Y$
    and $Y'$ is feasible for~\eqref{eq:weighted_dsw_min_Y_eq}.
    Therefore $\eqref{eq:weighted_dsw_min_Y_eq} \le \eqref{eq:weighted_dsw_min_Y}$.

    Let $Y$ be feasible for~\eqref{eq:weighted_dsw_min_Y_eq}.
    Define $Y' = (I_A \ot \sqrt{W^T}) Y (I_A \ot \sqrt{W^T})^\dag$.
    Then $\Tr_\hA Y' = \sqrt{W^T} (\Tr_\hA Y) \sqrt{W^T}^\dag = \lambda W^T$.
    And $Y' \ge (I_A \ot \sqrt{W^T}) \proj{\rw} (I_A \ot \sqrt{W^T})^\dag =
    \ket{\vw}\bra{\vw}$.
    So $Y'$ is feasible for~\eqref{eq:weighted_dsw_min_YWT_eq}, giving
    $\eqref{eq:weighted_dsw_min_YWT_eq} \le \eqref{eq:weighted_dsw_min_Y_eq}$.

    Since~\eqref{eq:weighted_dsw_min_YWT} has a larger feasible region,
    $\eqref{eq:weighted_dsw_min_YWT} \le \eqref{eq:weighted_dsw_min_YWT_eq}$.

    Let $Y$ be feasible for~\eqref{eq:weighted_dsw_min_YWT} and let
    $V$ be the inverse of $\sqrt{W^T}$, or the pseudoinverse if $W^T$ is singular.
    Set $Y' = (I_A \ot V) Y (I_A \ot V)^\dag$.
    Then $\Tr_\hA Y \le \lambda V W^T V^\dag \le \lambda I_\hB$
    (the last being equality if $W^T$ is not singular).
    And $Y' \ge (I_A \ot V) \ket{\vw}\bra{\vw} (I_A \ot V)^\dag = \proj{\rw}$.
    So $Y'$ is feasible for~\eqref{eq:weighted_dsw_min_Y} giving
    $\eqref{eq:weighted_dsw_min_Y} \le \eqref{eq:weighted_dsw_min_YWT}$.
    Therefore $\eqref{eq:weighted_dsw_min_Y}
    = \eqref{eq:weighted_dsw_min_Y_eq}
    = \eqref{eq:weighted_dsw_min_YWT_eq}
    = \eqref{eq:weighted_dsw_min_YWT}$.

    Suppose $W$ is non-singular and let $Y$ be feasible for~\eqref{eq:weighted_dsw_min_YWT}.
    Then $Y' = (I \ot W^{-T}) Y (I \ot W^{-T})^\dag$ is feasible
    for~\eqref{eq:weighted_dsw_min_YWinvT}.
    Conversely, if $Y'$ is feasible for~\eqref{eq:weighted_dsw_min_YWinvT} then
    $Y = (I \ot W^T) Y' (I \ot W^T)^\dag$ is feasible
    for~\eqref{eq:weighted_dsw_min_YWT}.
    Therefore $\eqref{eq:weighted_dsw_min_YWinvT} = \eqref{eq:weighted_dsw_min_YWT}$

    Let $T$ be feasible for~\eqref{eq:weighted_dsw_max_opnorm} with value $\lambda$.
    Let $\ket{\psi}$ be the normalized vector achieving
    \begin{align}
        \lambda &= \braopket{\psi}{ (\sqrt{W} \ot I)(T + I \ot I)(\sqrt{W} \ot I) }{\psi}.
        \label{eq:dsw_opnorm_psi}
    \end{align}
    Let $\rho$ be such that $(I \ot \sqrt{\rho}) \ket{\Phi} = \ket{\psi}$.
    This requires that $\ket{\psi}$, seen as an operator
    $\linop{\hilbdual{\hB} \to \hA}$ (under the isomorphism
    between $\hA$ and $\hilbdual{\hB}$ induced by $\ket{\Phi}$), is positive semidefinite.
    This is always achievable because~\eqref{eq:weighted_dsw_max_opnorm} is invariant
    under unitary transform on the $\hB$ side of $T$.
    Note that $\braket{\psi}{\psi} = 1$ gives $\Tr \rho = 1$.

    Define $T' = (I \ot \sqrt{\rho})^\dag T (I \ot \sqrt{\rho})$.
    Then $T' \in S^\perp \ot \linop{\hB}$ and
    \begin{align}
        \lambda
        &= \braopket{\Phi}{ (\sqrt{W} \ot \sqrt{\rho})
            (T + I \ot I)(\sqrt{W} \ot \sqrt{\rho}) }{\Phi}
        \\ &= \braopket{\rw}{(I \ot \sqrt{\rho})(T + I \ot I)(I \ot \sqrt{\rho})}{\rw}
        \\ &= \braopket{\rw}{T' + I \ot \rho}{\rw}
    \end{align}
    So $T'$ is feasible for~\eqref{eq:weighted_dsw_max_T} with value $\lambda$, giving
    $\eqref{eq:weighted_dsw_max_T} \ge \eqref{eq:weighted_dsw_max_opnorm}$.

    To show
    $\eqref{eq:weighted_dsw_max_opnorm} \ge \eqref{eq:weighted_dsw_max_T}$
    run this proof in reverse, starting with $T'$ being feasible
    for~\eqref{eq:weighted_dsw_max_T}.
    There are a couple bumps in this road.
    First, $\ket{\psi}$ may not be the eigenvector for the largest eigenvalue
    in~\eqref{eq:dsw_opnorm_psi}.
    This is not a significant issue because we only seek
    $\eqref{eq:weighted_dsw_max_opnorm} \ge \lambda$.
    The second issue is that finding $T$ satisfying
    $T' = (I \ot \sqrt{\rho})^\dag T (I \ot \sqrt{\rho})$
    requires that the null space of $T'$ contains the null space
    of $I \ot \rho$.  And indeed this is the case.  Let $P$ be the projector onto
    the null space of $\rho$.
    Then
    \begin{align}
        T' + I \ot \rho \ge 0
        &\implies (I \ot P)^\dag (T' + I \ot \rho) (I \ot P) \ge 0
        \\ &\implies (I \ot P)^\dag T' (I \ot P) \ge 0
    \end{align}
    Since $I \in S$ and $T' \in S^\perp \ot \linop{\hA}$, we have $\Tr_\hA T' = 0$.
    Therefore
    \begin{align}
        0 &= \Tr\left( P^\dag (\Tr_\hA T') P \right)
        \\ &= \Tr\left( (I \ot P)^\dag T' (I \ot P) \right)
    \end{align}
    Any positive semidefinite operator with vanishing trace vanishes so
    $(I \ot P)^\dag T' (I \ot P) = 0$; the null space of $T'$ contains the null space
    of $\rho$.

    The equivalence of~\eqref{eq:weighted_dsw_max_opnorm}
    and~\eqref{eq:weighted_dsw_max_opnorm2}
    follows directly from the relation
    $\opnorm{\sqrt{A} B \sqrt{A}} = \opnorm{\sqrt{B} A \sqrt{B}}$,
    valid for all $A, B \ge 0$.

    Any solution to \eqref{eq:weighted_dsw_max_opnorm} can be transformed into a solution to
    \eqref{eq:weighted_dsw_max_Y} of the same value by taking $Y = n^{-1} (T + I \ot I)$.
    Conversely, if $Y$ is a solution to \eqref{eq:weighted_dsw_max_Y} we can take
    $T = n Y - I \ot I$.  Note that $\TrA T = 0$ so
    $Y \in (S^\perp + \mathbb{C}I) \ot \linop{\hB}$ implies $T \in S^\perp \ot \linop{\hB}$.
    Therefore $\eqref{eq:weighted_dsw_max_Y} = \eqref{eq:weighted_dsw_max_opnorm}$.

    Similarly, any solution to~\eqref{eq:weighted_dsw_max_T} can be transformed into
    a solution to~\eqref{eq:weighted_dsw_max_Y_v2} by defining $Y = n^{-1} (T + I \ot \rho)$.
    Any solution to~\eqref{eq:weighted_dsw_max_Y_v2} can be transformed into a solution
    to~\eqref{eq:weighted_dsw_max_T} by defining $\rho = \TrA Y$ and
    $T = n Y - I \ot \rho$.
    We have $\TrA T = n \TrA Y - n \rho = 0$ so
    $T \in (\mathbb{C}I)^\perp \ot \linop{\hB}$.
    Since $Y \in (S^\perp + \mathbb{C}I) \ot \linop{\hB}$ we have
    $T \in S^\perp \ot \linop{\hB}$.
    Therefore $\eqref{eq:weighted_dsw_max_Y_v2} = \eqref{eq:weighted_dsw_max_T}$.
\end{proof}

The weighted $\dsw$ satisfies all the same basic properties of the classical $\lov$.

\begin{theorem}
    For matrices $W,X \in \psd{\hA}, W' \in \psd{\hA'}$, scalar $\alpha \ge 0$,
    and non-commutative graphs $S \subseteq \linop{\hA}, S' \subseteq \linop{\hA'}$
    the following basic properties hold:
    \begin{thmlist}
        \item $S \subseteq S' \implies \dsw(S, W) \ge \dsw(S', W)$
            \label[theorem]{thm:dsw_monotoneS}
        \item $W \le X \implies \dsw(S, W) \le \dsw(S, X)$
            \label[theorem]{thm:dsw_monotoneW}
        \item $\dsw(S, \alpha W) = \alpha \dsw(S, W)$
            \label[theorem]{thm:dsw_linear}
        \item $\dsw(S, W + X) \le \dsw(S, W) + \dsw(S, X)$
            \label[theorem]{thm:dsw_subadditive}
        \item $\dsw(\mathbb{C}I, W) = \dim(\hA) \Tr W$
            \label[theorem]{thm:dsw_empty}
        \item $\dsw(\linop{\hA}, W) = \opnorm{W}$
            \label[theorem]{thm:dsw_full}
        \item $\opnorm{W} \le \dsw(S, W) \le \dim(\hA) \Tr W$
            \label[theorem]{thm:dsw_bounded}
        \item $\dsw(S \ot S', W \ot W') = \dsw(S * S', W \ot W') = \dsw(S, W) \dsw(S', W')$
            where $S * S' = (\vecmod{S}{\mathbb{C}I}) \ot \linop{\hA'}
            + \linop{\hA} \ot (\vecmod{S'}{\mathbb{C}I}) + \mathbb{C}I \ot I$.
            \label[theorem]{thm:dsw_prod}
    \end{thmlist}
    Note that $S \ot S'$ is analogous to the strong product for classical graphs and
    $S * S'$ is analogous to the disjunctive product.
\end{theorem}
\begin{proof}
    \itemref{thm:dsw_monotoneS}:
        In~\eqref{eq:weighted_dsw_min_Y} the feasible set for
        $\dsw(S, W)$ is contained in the feasible set for $\dsw(S', W)$.

    \itemref{thm:dsw_monotoneW}:
        The objective function in~\eqref{eq:weighted_dsw_max_opnorm2}
        is monotone in $W$.

    \itemref{thm:dsw_linear}:
        The objective function in~\eqref{eq:weighted_dsw_max_opnorm2}
        is linear in $W$.

    \itemref{thm:dsw_subadditive}:
        The objective function in~\eqref{eq:weighted_dsw_max_opnorm2}
        is subadditive in $W$.

    \itemref{thm:dsw_empty}:
        Consider~\eqref{eq:weighted_dsw_min_Y_eq}.
        The conditions $Y \in S \ot \linop{\hB}$ and $\TrA Y = \lambda I$ force
        $Y = \lambda n^{-1} I \ot I$ where $n = \dim(\hA)$.
        The operator norm of $\proj{\rw}$ is
        $\braket{\rw}{\rw} = \braopket{\Phi}{W \ot I}{\Phi} = \Tr W$.
        So $Y \ge \proj{\rw} \iff \lambda n^{-1} \ge \Tr W$.
        So $\min \lambda = n \Tr W$.

    \itemref{thm:dsw_full}:
        Consider~\eqref{eq:weighted_dsw_min_Y}.
        $Y = \proj{\rw}$ is feasible and there is no more optimal solution due to the
        constraint $Y \ge \proj{\rw}$.  We have $\TrA Y = W^T$ so the optimal
        value is $\lambda = \opnorm{W^T} = \opnorm{W}$.

    \itemref{thm:dsw_bounded}:
        Since $\mathbb{C}I \subseteq S \subseteq \linop{\hA}$ this follows
        from~\itemref{thm:dsw_monotoneS},~\itemref{thm:dsw_empty},
        and~\itemref{thm:dsw_full}.

    \itemref{thm:dsw_prod}:
        Suppose $S \subseteq \linop{\hA}$ and $S' \subseteq \linop{\hA'}$.
        Let $Y,\lambda$ and $Y',\lambda'$ be optimal for~\eqref{eq:weighted_dsw_min_Y_eq}
        so $\lambda = \dsw(S, W)$ and $\lambda' = \dsw(S', X)$.
        Then $Y \ot Y' \in \linop{\hA \ot \hA' \ot \hB \ot \hB'}$ is feasible
        for~\eqref{eq:weighted_dsw_min_Y_eq} for $\dsw(S \ot S', W \ot X)$ with
        value $\lambda \lambda'$, giving
        $\dsw(S \ot S', W \ot X) \le \dsw(S, W) \dsw(S', X)$.

        Since $S \ot S' \subseteq S * S'$,~\itemref{thm:dsw_monotoneS}
        gives $\dsw(S * S', W \ot X) \le \dsw(S \ot S', W \ot X)$.

        Let $Y$ and $Y'$ be optimal for~\eqref{eq:weighted_dsw_max_Y}.
        Then $Y \in (S^\perp + \mathbb{C}I) \ot \linop{\hB} = (\vecmod{S}{\mathbb{C}I})^\perp \ot \linop{\hB}$.
        Similarly, $Y' \in (\vecmod{S'}{\mathbb{C}I})^\perp \ot \linop{\hB'}$.
        So
        \begin{align}
        Y \ot Y' &\in
            (\vecmod{S}{\mathbb{C}I})^\perp \ot (\vecmod{S'}{\mathbb{C}I})^\perp \ot \linop{\hB \ot \hB'}
        \\ &= (\vecmod{S}{\mathbb{C}I} \ot \linop{\hA'} + \linop{\hA} \ot \vecmod{S'}{\mathbb{C}I})^\perp \ot \linop{\hB \ot \hB'}
        \\ &= ((S * S')^\perp + \mathbb{C}I \ot I) \ot \linop{\hB \ot \hB'}.
        \end{align}
        And $\Tr_{\hA \ot \hA'}(Y \ot Y') = (\TrA Y)(\Tr_{\hA'} Y') = I_{\hB \ot \hB'}$ so $Y \ot Y'$ is feasible
        for~\eqref{eq:weighted_dsw_max_Y} for $\dsw(S \ot S', W \ot W)$.
        Since operator norm is multiplicative under tensor product and
        $\sqrt{W \ot X} = \sqrt{W} \ot \sqrt{X}$,
        the value of this solution is $\dsw(S, W) \dsw(S', X)$.
        Therefore $\dsw(S \ot S', W \ot W) \ge \dsw(S, W) \dsw(S', X)$.
\end{proof}

\begin{theorem}
    $\dsw(S, W)$ is uniformly continuous in $W$.
    In fact, for $W, X \in \psd{\hA}$ we have $\abs{\dsw(S, W) - \dsw(S, X)} \le n \trnorm{W - X}$
    where $n = \dim(\hA)$.
    \label{thm:dsw_continuous}
\end{theorem}
\begin{proof}
    Let $W, X \ge 0$.
    Using \cref{thm:dsw_monotoneW,thm:dsw_subadditive,thm:dsw_bounded}, we have
    \begin{align}
        \dsw(S, W) &= \dsw(S, X + (W-X))
        \\ &\le \dsw(S, X + \abs{W-X})
        \\ &\le \dsw(S, X) + \dsw(S, \abs{W-X})
        \\ &\le \dsw(S, X) + n \trnorm{W - X}.
    \end{align}
    Therefore $\dsw(S, W) - \dsw(S, X) \le n \trnorm{W - X}$.
    Similar logic with $W$ and $X$ swapped yields
    $\dsw(S, X) - \dsw(S, W) \le n \trnorm{W - X}$.
\end{proof}

For classical graphs there is a formulation of $\lov(G, w)$ which makes
clear that $\{ w \ge 0 : \lov(G, w) \le 1 \}$ is in fact a spectrahedral
shadow~\cite[section 29]{knuth94}.
As a side note, this form tends to be more powerful for adding extra constraints~\cite{Galli2017}.
\begin{align}
    \lov(G, w) = \min\left\{ \lambda :
        \left[
        \begin{array}{c|c}
            \lambda & \bra{w} \\
            \hline
            \ket{w} & Z
        \end{array}
        \right] \ge 0,
        \diag(Z)=w, Z_{ij}=0 \textrm{ for } i \ne j, i \nsim j
    \right\}
\end{align}

Something similar can be done for non-commutative graphs.  Note this optimization is over
$(n^2+1) \times (n^2+1)$ matrices where $n = \dim(\hA)$.
\begin{theorem}
    \begin{align}
        \dsw(S, W) = \min\left\{ \lambda :
            \left[
                \begin{array}{c|c}
                    \lambda & \bra{\vw} \\
                    \hline
                    \ket{\vw} & Z
                \end{array}
            \right] \ge 0,
            Z \in S \ot \linop{\hB},
            \TrA Z = W^T
        \right\}.
        \label{eq:dsw_grotschel}
    \end{align}
    \label{thm:dsw_grotschel}
\end{theorem}
\begin{proof}
    Taking~\eqref{eq:weighted_dsw_min_YWT_eq} from \cref{thm:weighted_dsw_equiv_forms}
    and defining $Z = \lambda^{-1} Y$ gives
    \begin{align}
        \dsw(S, W)
        &= \min\{ \lambda :
            Z \in S \ot \linop{\hB},
            \TrA Z = W^T,
            Z - \lambda^{-1} \ket{\vw}\bra{\vw} \ge 0 \}.
    \end{align}
    But $Z - \lambda^{-1} \ket{\vw}\bra{\vw}$ is the Schur complement of the block matrix
    in~\eqref{eq:dsw_grotschel}, so its positive semidefiniteness is equivalent to the
    positive semidefiniteness of that block matrix.
\end{proof}


Fundamental to the theory of weighted $\lov$ for classical graphs is the
\emph{theta body}.

\begin{definition}
    The \emph{theta body} of a graph $G$, denoted $\thbody(G)$, is a set of entrywise non-negative
    vectors given by the following equivalent definitions.
    \begin{align}
        \thbody(G) &= \{ x \ge 0 : \lov(\Gc, x) \le 1 \}
        \label{eq:thbody1}
        \\ &= \{ x \ge 0 : \braket{y}{x} \le 1 \textrm{ for all } y \ge 0, \lov(G, y) \le 1 \}.
        \label{eq:thbody2}
    \end{align}
\end{definition}

We define this for non-commutative graphs by extending~\eqref{eq:thbody2} rather
than~\eqref{eq:thbody1} because graph complement is more subtle for non-commutative graphs.
(Complements of non-commutative graphs will be explored in the following sections.)
Note that~\cite{btw2019}, having defined a different theta, defines a different theta body,
which we will not be investigating here.

\begin{definition}
    The \emph{theta body} for a non-commutative graph $S$ is
    \begin{align}
        \dswbody(S) &= \{ X \in \psd{\hA} : \Tr(X W) \le 1 \textrm{ for all }
            W \in \psd{\hA}, \dsw(S, W) \le 1 \}.
    \end{align}
    \label{def:dswbody}
\end{definition}

Though we define $\dswbody(S)$ above, we will generally be more interested in its
anti-blocker, $\dswbodyS(S)$.  This in fact has a simpler definition as we shall now see.

\begin{theorem}
    The theta body of $S$ satisfies the following basic properties.
    \begin{thmlist}
        \item $\dswbody(S)$ is a convex corner.
            \label[theorem]{thm:dswbody_convex}
        \item $\dswbodyS(S) = \{ W \in \psd{\hA} : \dsw(S, W) \le 1 \}$.
            \label[theorem]{thm:dswbody_abl}
        \item $\dsw(S, W) = \max\{ \Tr(X W) : X \in \dswbody(S) \}$.
            \label[theorem]{thm:dsw_from_dswbody}
    \end{thmlist}
\end{theorem}
\begin{proof}
    \itemref{thm:dswbody_convex}:
        $\dswbody(S) = \{ W \in \psd{\hA} : \dsw(S, W) \le 1 \}^\sharp$.
        By \cref{thm:antiblocker_is_corner}, the anti-blocker of any non-empty subset of
        $\psd{\hA}$ is a convex corner.

    \itemref{thm:dswbody_abl}:
        Define $\cornC = \{ W \in \psd{\hA} : \dsw(S, W) \le 1 \}$.
        From the monotonicity and convexity of $\dsw$
        (\cref{thm:dsw_monotoneW,thm:dsw_subadditive})
        it follows that $\cornC$ is a convex corner.
        And $\cornC^\sharp = \dswbody(S)$ follows directly from
        \cref{def:dswbody} and the definition of anti-blocker.
        By the second anti-blocker theorem (\cref{thm:second_antiblocker}),
        $\cornC = \cornC^{\sharp\sharp} = \dswbodyS(S)$.

    \itemref{thm:dsw_from_dswbody}:
        For $W \in \psd{\hA}$ we have
        \begin{align}
            \max\{ \Tr(X W) : X \in \dswbody(S) \} \le 1
            &\iff W \in \dswbodyS(S)
            \\ &\iff \dsw(S, W) \le 1.
        \end{align}
        The first implication follows from the definition of anti-blocker,
        the second from~\itemref{thm:dswbody_abl}.
        Then~\itemref{thm:dsw_from_dswbody} follows from linearity of $\dsw$.
\end{proof}

Though we've formed the definitions in terms of $\dswbody(S)$ for historical reasons,
we will generally find more use for $\dswbodyS(S)$.
The following theorem shows that $\dswbodyS(S)$ is a spectrahedral shadow.

\begin{theorem}
    \begin{align}
        \dswbodyS(S) &= \left\{ W \in \psd{\hA} :
            \left[
                \begin{array}{c|c}
                    1 & \bra{\vw} \\
                    \hline
                    \ket{\vw} & Z
                \end{array}
            \right] \ge 0,
            Z \in S \ot \linop{\hB},
            \TrA Z = W^T
        \right\}.
        \label{eq:dsw_corner_grotschel}
    \end{align}
\end{theorem}
\begin{proof}
    Follows from \cref{thm:dsw_grotschel} and \cref{thm:dswbody_abl}.
\end{proof}

It is worth noting that not only is $\dswbodyS(S)$ a convex corner, it is in fact available
as an SDP subroutine.  That is, the variables and constraints from~\eqref{eq:dsw_corner_grotschel} can
be part of a larger SDP.  For example, we could find the $W \in \dswbodyS(S)$ that maximizes the
inner product $\Tr(W X)$ for some given $X$, the subject of~\cref{sec:holder}.
Or the same but constraining $W$ to be block diagonal.
Or something more exotic such as the entropy of this convex corner.  Entropy of convex corners is
discussed in~\cite[section 2.4.3]{borelandthesis} and computing entropy via SDP in~\cite{Fawzi2018}.
Such numerical investigations have been instrumental in discovering many of the theorems in this
paper.


\section{Duality}
\label{sec:holder}

Classical $\lov$ functions satisfy an interesting duality relation~\cite{Grtschel1986}.
For $w, x \ge 0$ (entrywise positive vectors),
\begin{align}
    \lov(G, w) \lov(\Gc, x) \ge \braket{w}{x}
    \label{eq:classical-ineq}
\end{align}
and for every $w$ there is some $x$ that saturates this inequality.
Equivalently,
\begin{align}
    \lov(\Gc, x) = \max\{ \braket{w}{x} : \lov(G, w) \le 1 \}.
\end{align}
And in terms of theta bodies,
\begin{align}
    \thbody(\Gc) = \thbody^\flat(G)
\end{align}
where $\flat$ is the classical anti-blocker,
$\thbody^\flat(G) = \{ w \ge 0 : \braket{w}{x} \le 1 \textrm{ for all } x \in \thbody(G) \}$

The goal of this section is an analogous theorem for non-commutative graphs:
\begin{theorem}
    Let $S \in \linop{\hA}$ be a non-commutative graph and $n = \dim(\hA)$.
    For any $W, X \in \psd{\hA}$,
    \begin{align}
        \dsw(S, W) \dsw(S^\perp + \mathbb{C}I, X) \ge n \Tr(W X)
    \end{align}
    and for every $W$ there is some $X$ that saturates this inequality.
    Equivalently,
    \begin{align}
        \dsw(S^\perp + \mathbb{C}I, X) =
        \max\{ n \Tr(W X) : W \ge 0, \dsw(S, W) \le 1 \}.
        \label{eq:holder_dsw}
    \end{align}
    In terms of theta bodies,
    \begin{align}
        \dswbodyS(S^\perp + \mathbb{C}I) = n^{-1} \dswbody(S)
        \label{eq:holder_dsw_corner}
    \end{align}
    \label{thm:holder_dsw}
\end{theorem}

Notice the extra factor $n$ compared to~\eqref{eq:classical-ineq}.
Although our weighted thetas match the classical definition for classical
graphs (\cref{thm:match_classical}), our graph complement is less dense and therefore
yields a larger $\dsw$.
Specifically, the classical complement would be
$\linspan\{ S^\perp, \ket{i}\bra{i} : i \in \{1,\dots,n\} \}$
whereas we use here $\linspan\{S^\perp, I\}$.
In~\cref{sec:block_graphs} we will investigate this further, showing that the diagonal elements
added (or, in this case, not added) into the graph complement are responsible for this scaling factor.

Before proving~\cref{thm:holder_dsw} we record a simple lemma.

\begin{lemma}
    For $M \ge 0$, the following are equivalent:
    \begin{thmlist}
        \item $M \ge \ket{x}\bra{x}$
        \item There is some $\ket{h}$ such that $\braket{h}{h} \le 1$ and
            $\sqrt{M} \ket{h}\bra{h} \sqrt{M} = \ket{x}\bra{x}$.
    \end{thmlist}
    \label{thm:handle}
\end{lemma}
\begin{proof}
    The second implies the first because
    $I \ge \ket{h}\bra{h} \implies M \ge \sqrt{M} \ket{h}\bra{h} \sqrt{M}$.
    On the other hand if the first is true then take
    $\ket{h} = L \ket{x}$ where $L$ is the pseudoinverse of $\sqrt{M}$.
    Then $\sqrt{M} \ket{h} = P \ket{x}$ where $P$ is the projector onto the support
    of $M$.  But $M \ge \ket{x}\bra{x}$ requires $\ket{x}$ to be in the support of
    $M$ so $P \ket{x} = \ket{x}$.

\end{proof}

%

\begin{proof}[Proof of~\cref{thm:holder_dsw}]
    Fix $X \ge 0$ and consider the optimization in~\eqref{eq:holder_dsw}.
    For brevity, define the set
    $F = \{ Y \in Y \in S \ot \linop{\hB} : Y \ge 0, \TrA Y = I \}$.
    Using~\eqref{eq:weighted_dsw_min_Y_eq},
    \begin{align}
        \max\{ \Tr(W X) &: W \ge 0, \dsw(S, W) \le 1 \}
        \\ &= \max\left\{ \Tr(W X) : W \ge 0, Y \in F, Y \ge \proj{\rw} \right\}
        \\ &= \max\left\{ \braopket{\rw}{X \ot I}{\rw} : W \ge 0, Y \in F, Y \ge \proj{\rw} \right\}
        \label{eq:max_W_XI_W}
    \end{align}
    Using \cref{thm:handle},
    \begin{align}
        \eqref{eq:max_W_XI_W}
        &= \max\left\{ \braopket{\rw}{X \ot I}{\rw} : W \ge 0, Y \in F, \braket{h}{h} \le 1,
            \sqrt{Y} \ket{h} = \ket{\rw} \right\}
        \\ &= \max\left\{ \braopket{h}{\sqrt{Y} (X \ot I) \sqrt{Y}}{h} : \braket{h}{h} \le 1, Y \in F \right\}
            \label{eq:trace_yYhandle}
    \end{align}
    For the equality in~\eqref{eq:trace_yYhandle} we make use of unitary freedom on $\linop{\hB}$.
    That is, for any feasible solution to~\eqref{eq:trace_yYhandle}, to get a feasible
    solution to the prior equation we need that $\sqrt{Y}\ket{h} = \ket{\rw}$ for some
    $W$.  This requires that $\sqrt{Y}\ket{h}$, viewed as an operator
    $\linop{\hilbdual{\hB} \to \hA}$ (under the isomorphism
        between $\hA$ and $\hilbdual{\hB}$ induced by $\ket{\Phi}$), is positive semidefinite.
    That can be achieved by applying some unitary on the $\hB$ side
    of $Y$, which is allowed because the condition $Y \in F$ allows that unitary degree of
    freedom.
    Continuing,
    \begin{align}
        \eqref{eq:trace_yYhandle}
        &= \max\left\{ \opnorm{\sqrt{Y}(X \ot I)\sqrt{Y}} : Y \in F \right\}
        \\ &= \max\left\{ \opnorm{(\sqrt{X} \ot I) Y (\sqrt{X} \ot I)} : Y \in F \right\}
            \label{eq:max_sqrty_Y}
    \end{align}
    Using~\eqref{eq:weighted_dsw_max_Y} from \cref{thm:weighted_dsw_equiv_forms}
    and the fact that $Y \in S \ot \linop{\hB} \iff Y \in ((S^\perp + \mathbb{C}I)^\perp + \mathbb{C}I) \ot \linop{\hB}$
    we have $\eqref{eq:max_sqrty_Y} = n^{-1} \dsw(S^\perp + \mathbb{C}I, X)$.  Therefore~\eqref{eq:holder_dsw} holds.

    As for~\eqref{eq:holder_dsw_corner}, we have
    \begin{align}
        \dswbodyS(S^\perp + \mathbb{C}I)
        &= \{ X \ge 0 : \dsw(S^\perp + \mathbb{C}I, X) \le 1 \}
        \\ &= \{ X \ge 0 : \max\{ n \Tr(W X) : W \ge 0, \dsw(S, W) \le 1 \} \le 1 \}
        \\ &= \{ X \ge 0 : n \Tr(W X) \le 1 : \forall W \ge 0, \dsw(S, W) \le 1 \}
        \\ &= \{ n^{-1} X \ge 0 : \Tr(W X) \le 1 : \forall W \ge 0, \dsw(S, W) \le 1 \}
        \\ &= n^{-1} \dswbody(S)
    \end{align}
    Taking the anti-blocker of both sides and applying the second anti-blocker theorem
    yields~\eqref{eq:holder_dsw_corner}.
\end{proof}

As an alternate proof, we could have modified the SDP from \cref{thm:dsw_grotschel}, setting
$\lambda = 1$ and maximizing $\Tr(WX)$.
The dual of this program, after some manipulation, should then yield an SDP of the form in
\cref{thm:dsw_grotschel} for $S^\perp + \mathbb{C}I$.
Though this proof ends up being more difficult, it does have the advantage of being easy to
verify numerically: just run it in an SDP solver.


\section{Graphs with a block structure}
\label{sec:block_graphs}

The main result of the previous section, \cref{thm:holder_dsw}, is not a true generalization
of the classical $\vartheta$ duality relation~\eqref{eq:classical-ineq}.
It has two problems: an extra factor of $n$ and a different type of graph complement.
The purpose of this section is to remedy both of these problems, which as we will see are related.
If $G$ is a classical graph, $\SG = \linspan\{ \ket{i}\bra{j} : i \simeqG j \}$ is the corresponding
non-commutative graph.
The non-commutative graph for $\Gc$ is then $\SGc = \SG^\perp + \diags$ where $\diags$
is the space of diagonal matrices.
This differs from the inverse $S^\perp + \mathbb{C}I$ used in \cref{thm:holder_dsw}.
In general, how are we to know whether $S^\perp + \mathbb{C}I$ or $S^\perp + \diags$
is the appropriate complement?  Or something different from either of these?

The answer comes from~\cite{arxiv:1002.2514}, the extended arXiv version of~\cite{dsw2013}.
The non-commutative graph associated with a classical graph has a special linear algebraic structure.
We take the following definition from~\cite{arxiv:1002.2514} but omit discussion of the $S_0$-valued
inner product, which we will not need.

\begin{definition}
    For a Hilbert space $\hA$ and a $C^*$-algebra $S_0 \subseteq \linop{\hA}$,
    a non-commutative graph $S \subseteq \linop{\hA}$ is said to be an \emph{$S_0$-graph}
    if $S$ is an $S_0$ bimodule, i.e., $S_0 \subseteq S$ and $S_0 S S_0 = S$.
\end{definition}

The space $S_0$ can be thought of as the ``vertices'' of $S$.
For our purposes $S_0$ will be significant in a number of ways.
We will use it to define the graph complement, taking $\Sc = S^\perp + S_0$.
We will find that convex corners associated with $\dsw$ or with cliques end up taking maximal
values in $S_0'$, the commutant of $S_0$.
We will find the shape of $S_0$ to factor into a generalized version of \cref{thm:holder_dsw},
taking the place of the spurious factor of $n$ appearing in that theorem.

Note that if $S$ is an $S_0$-graph and $T_0$ is a subalgebra of $S_0$ then $S$ is also
a $T_0$-graph.  In particular, any $S_0$-graph is also a $\mathbb{C}I$-graph.
But when speaking of $S$ as being an $S_0$-graph or a $\mathbb{C}I$-graph,
the corresponding complement graphs will be different.

If $G$ is a classical graph and $\SG = \linspan\{ \ket{i}\bra{j} : i \simeqG j \}$
then $\SG$ is a $\diags$-graph with $\diags = \{ \ket{i}\bra{i} : i \in V(G) \}$.
And $\SGc = \SG^\perp + \diags$ is also a $\diags$ graph.
In general we will take the complement of an $S_0$-graph to be
$S^\perp + S_0$.

\begin{theorem}
    If $S$ is an $S_0$-graph, then $\Sc = S^\perp + S_0$ is also an $S_0$-graph.
\end{theorem}
\begin{proof}
    Clearly $S_0 \in \Sc$.  It remains to show $S_0 \Sc S_0 = \Sc$.
    Suppose $X \in S^\perp$ and $K,L \in S_0$.  We will show $K X L \in S^\perp$.
    For any $Y \in S$ we have
    \begin{align}
        \Tr(Y^\dag K X L)
        &= \Tr(L Y^\dag K X)
        \\ &= \Tr((K^\dag Y L^\dag)^\dag X).
    \end{align}
    But $K^\dag, L^\dag \in S_0$ and $S_0 S S_0 = S$ so $K^\dag Y L^\dag \in S$.
    Since $X \in S^\perp$, the trace vanishes and we have $K X L \in S^\perp$.
    Therefore $S_0 S^\perp S_0 \subseteq S^\perp$.
    And $I \in S_0$ so $S^\perp = I S^\perp I \subseteq S_0 S^\perp S_0$.
    Thus we have $S^\perp = S_0 S^\perp S_0$ and
    \begin{align}
        S_0 \Sc S_0 &= S_0 (S^\perp + S_0) S_0
        \\ &=  S_0 S^\perp S_0 + S_0 S_0 S_0
        \\ &=  S^\perp + S_0.
    \end{align}
\end{proof}

Going forward, we will be using a particular decomposition of the $S_0$ space.
The structure theorem for a finite dimensional $C^*$-algebra $S_0 \subseteq \linop{\hA}$
gives~\cite{arxiv:1002.2514}
\begin{align}
    S_0 = \bigoplus_{i=1}^r \linop{\hAi} \ot I_{\hYi}
    \qquad \textrm{with} \qquad
    \hA = \bigoplus_{i=1}^r \hAi \ot \hYi.
    \label{eq:vertex_algebra}
\end{align}

For a $C^*$-algebra $S_0 \subseteq \linop{\hA}$, we denote by $S_0'$ it's commutant:
the set of operators in $\linop{\hA}$ commuting with every element of $S_0$.
This is also a $C^*$-algebra, and takes the form
\begin{align}
    S_0' = \bigoplus_{i=1}^r I_\hAi \ot \linop{\hYi}.
    \label{eq:commutator_structure}
\end{align}
The commutant will play a significant role as we move forward:
we will show the maximal values of $\dswbodyS(S)$ are in $S_0'$.
Classical graphs have the peculiar property that $S_0' = S_0$.

The duality relations we will find for $S_0$ graphs have a more complicated structure, requiring the
use of various projection and scaling operators.

\begin{definition}
    Let $S$ be an $S_0$-graph, with $S_0$ decomposed as in~\eqref{eq:vertex_algebra}.
    Let $P_i$ be the projector onto $\hAi \ot \hYi$.
    For $W \in \linop{\hA}$ define
    \begin{align}
        \commproj(W) &= \sum_i \dimAi^{-1} I_\hAi \ot \TrAi (P_i W P_i)
        \label{eq:def_commproj}
        \\ \blockscale(W) &= \sum_i \dimYi^{-1} I_\hAi \ot \TrAi (P_i W P_i)
        \\ \diagscale &= \sum_i \dimAi^{-1} \dimYi P_i.
    \end{align}
    \label{def:blockutils}
\end{definition}

\begin{lemma}
    The entities from \cref{def:blockutils} satisfy the following basic facts.
    \begin{thmlist}
        \item $\blockscale$ and $\commproj$ are completely positive superoperators.
            \label[theorem]{thm:comproj_cp}
        \item $\commproj(W)$ is the projector onto the subspace $S_0'$, and in particular if
            $W \in S_0'$ then $\commproj(W) = W$.
            \label[theorem]{thm:comproj_projection}
        \item If $X \in S_0'$ then
            $\blockscale(X W X^\dag) = X \blockscale(W) X^\dag$ and
            $\commproj(X W X^\dag) = X \commproj(W) X^\dag$.
            \label[theorem]{thm:blockscale_commutes}
        \item $\commproj(W) = \blockscale(\sqrt{\diagscale} W \sqrt{\diagscale})$
            and $\blockscale(W) = \commproj(D^{-1/2} W D^{-1/2})$.
            \label[theorem]{thm:comproj_blockscale}
        \item $S_0'^\perp$ is in the null space of $\blockscale$.
            \label[theorem]{thm:blockscale_nullspace}
    \end{thmlist}
\end{lemma}
\begin{proof}
    \itemref{thm:comproj_cp}: Conjugation by $P_i$, partial trace, and direct product with a positive
    semidefinite operator are all completely positive operations; the composition and sum of these
    operations is also completely positive.

    \itemref{thm:comproj_projection}:
    That $\commproj(W) = W$ for all $W \in S_0'$ is clear by inspection.  On the other hand,
    suppose $W \in S_0'^\perp$.  Since $I_\hAi \ot M \in S_0'$ for any $1 \le i \le r$ and
    $M \in \linop{\hYi}$ (taking the other terms of the direct sum
    in~\eqref{eq:commutator_structure} to be zero), we have
    \begin{align}
        0 &= \Tr( (I_\hAi \ot M) P_i W P_i )
        \\ &= \Tr( M \TrAi(P_i W P_i) ).
    \end{align}
    Therefore all terms of~\eqref{eq:def_commproj} vanish and $\commproj(W) = 0$.

    \itemref{thm:blockscale_commutes}:
    Let $X \in S_0'$.
    Then $X = \oplus_i( I_\hAi \ot X_i )$ with $X_i \in \linop{\hYi}$.  We have
    \begin{align}
        \blockscale(X W X^\dag) &= \sum_i \dimYi^{-1} I_\hAi \ot \TrAi (P_i X W X^\dag P_i).
        \\ &= \sum_i \dimYi^{-1} I_\hAi \ot
            X_i \TrAi (P_i W P_i) X_i
        \\ &= X \left( \sum_i \dimYi^{-1} I_\hAi \ot
            \TrAi (P_i W P_i) \right) X
        \\ &= X \blockscale(W) X
    \end{align}
    The derivation for $\commproj(X W X^\dag) = X \commproj(W) X^\dag$ is analogous.

    \itemref{thm:comproj_blockscale}:
    Since $\commproj$ and $\blockscale$ differ only by scaling of the blocks, we have
    $\commproj(W) = \sqrt{\diagscale} \blockscale(W) \sqrt{\diagscale}$.
    And $\sqrt{D} \in S_0'$ so by \itemref{thm:blockscale_commutes} can be moved inside
    the parentheses, giving
    $\commproj(W) = \blockscale(\sqrt{\diagscale} W \sqrt{\diagscale})$.
    Since $D$ is invertible, we also have
    $\blockscale(W) = D^{-1/2} \commproj(W) D^{-1/2} = \commproj(D^{-1/2} W D^{-1/2})$.

    \itemref{thm:blockscale_nullspace}:
    Let $W \in S_0'^\perp$.
    Since $\commproj$ is the projector onto $S_0'$ we have $\commproj(W) = 0$.
    But as shown above,
    $\blockscale(W) = D^{-1/2} \commproj(W) D^{-1/2}$, so this vanishes as well.
\end{proof}

The following lemma explores the effect on $\dsw$ of removing the vertex set $S_0$
from $S$ and replacing it with the ``thin'' vertex set $\mathbb{C}I$.
Though it is of little interest on its own, this is the core technical lemma from which
we build the rest of the results of this section.

\begin{lemma}
    Let $S$ be an $S_0$-graph and take $\blockscale$ from \cref{def:blockutils}.
    Then
    \begin{align}
        \dsw(\vecmod{S}{S_0} + \mathbb{C}I, W) = \dsw(S, n \blockscale(W)).
        \label{eq:thin_diag}
    \end{align}
    In particular, $\dsw(\vecmod{S}{S_0} + \mathbb{C}I, W)$ is insensitive
    to any component of $W$ perpendicular to $S_0'$, the commutant of $S_0$.
    \label{thm:thin_diag}
\end{lemma}
\begin{proof}
    The proof details are tedious and are deferred to \cref{sec:main_proof}.
    Here we tie together the results from that appendix.

    If $W$ is not singular we have by~\cref{thm:mainlemma_le}
    $\dsw(\vecmod{S}{S_0} + \mathbb{C}I, W) \le \dsw(S, n \blockscale(W))$
    and by \cref{thm:mainlemma_ge}
    $\dsw(\vecmod{S}{S_0} + \mathbb{C}I, W) \ge \dsw(S, n \blockscale(W))$.
    This is extended to singular $W$ by continuity of $\dsw$,
    \cref{thm:dsw_continuous}.

    By \cref{thm:blockscale_nullspace}, $\blockscale(W)$
    is insensitive to any component of $W$ perpendicular to $S_0'$.
    Since $\dsw(S, \blockscale(W))$ is insensitive to this component,
    $\dsw(\vecmod{S}{S_0} + \mathbb{C}I, W)$ is as well.
\end{proof}

\begin{corollary}
    Let $G$ be a classical graph.  Let $S = \mathbb{C}I + \linspan\{ \ket{i}\bra{j} : i \simG j \}$.
    Then for any $W \in \psd{\hA}$,
    \begin{align}
        \dsw(S, W) = \abs{V(G)} \lov(G, \diag(W)).
    \end{align}
\end{corollary}

In terms of proof complexity, \cref{thm:thin_diag} is the main technical
result of this paper.
But this theorem tells us something about $\dsw(S, \blockscale(W))$
whereas what we really want is insight about $\dsw(S, W)$.
That is the goal of the remainder of this section.
First we need a couple lemmas about convex corners.

\begin{lemma}
    Every convex corner $\cornC$ has non-empty relative interior.
    And there is some $M \in \relint(\cornC)$ such that $P \cornC P = \cornC$ where $P$
    is the projector onto the support of $M$.
    \label{thm:corn_relint}
\end{lemma}
\begin{proof}
    Let $M \in \cornC \subseteq \linop{\hA}$ have the largest rank possible.
    If $M$ is full rank we are done, as $M/2$ is in the interior of $\cornC$.
    Otherwise, let $P$ be the projector onto the support of $M$.
    Note that $PXP = X$ for all $X \in \cornC$ because otherwise $(M+X)/2 \in \cornC$
    would contradict $M$ having the largest possible rank.
    Then $\{ X \in \linop{\hA} : PXP = X \}$ defines a subspace containing $\cornC$ and
    $M/2$ is in the interior of $\cornC$ relative to this subspace.
\end{proof}

\begin{lemma}
    Let $\cornC \subseteq \psd{\hA}$ be a convex corner
    and $\spaceD \subseteq \linop{\hA}$ be a subspace with $\spaceD^\dag = \spaceD$.
    Suppose the convex corner has no structure in directions perpendicular to $\spaceD$,
    \begin{align}
        W \in \cornC,
        E \in \spaceD^\perp,
        W + E \ge 0
        &\implies
        W + E \in \cornC
        \label{eq:condition_W_E}
    \end{align}
    Then $\cornC^\sharp = \her(\cornC^\sharp \cap \spaceD)$ where $\her$ is the hereditary
    closure~\eqref{eq:her_clos}.
    In other words, the maximal elements of $\cornC^\sharp$ are all in $\spaceD$
    (assuming $\cornC^\sharp$ is bounded).
    \label{thm:antiblocker_freespace}
\end{lemma}
\begin{proof}

    Let $X \in \cornC^\sharp$.  We will show there is a $Y \in \cornC^\sharp \cap \spaceD$
    such that $Y \ge X$.
    Since $X \in \cornC^\sharp$ we have $\max\{ \Tr(X^\dag W) : W \in \cornC \} \le 1$,
    or using~\eqref{eq:condition_W_E}, $p^* \le 1$ where
    \begin{align}
        p^* = \max\{ \Tr((W+E)^\dag X) : W \in \cornC, E \in \spaceD^\perp, W+E \ge 0 \}.
        \label{eq:W_E_primal}
    \end{align}
    This is a convex optimization problem on the convex domain
    $\mathcal{G} = \{ (W, E) : W \in \cornC, E \in \spaceD^\perp \}$.
    We will find the dual by the recipe in~\cite[chapter 5]{boyd2004convex}.
    The Lagrangian is
    \begin{align}
        L(W, E, Q)
        &= \Tr((W+E)^\dag X) + \Tr((W+E)^\dag Q)
        \\ &= \Tr(W^\dag (X+Q)) + \Tr(E^\dag (X+Q))
    \end{align}
    on the domain $(W, E) \in \mathcal{G}, Q \ge 0$.
    For a given $Q$ the Lagrange dual is the maximization of $L$ over the
    domain of the primal,
    \begin{align}
        g(Q) &= \max\{ L(W, E, Q) : (W, E) \in \mathcal{G} \}
        \\ &= \max\{ \Tr(W^\dag (X+Q)) + \Tr(E^\dag (X+Q)) : (W, E) \in \mathcal{G} \}
        \\ &=
        \begin{cases}
            \max\{ \Tr(W^\dag (X+Q)) : W \in \cornC \}
                & \textrm{ if } X+Q \in \spaceD
            \\ \infty & \textrm{ otherwise}.
        \end{cases}
    \end{align}
    The dual of~\eqref{eq:W_E_primal} is the minimization of this over $Q \ge 0$,
    \begin{align}
        d^* &= \min\{ g(Q) : Q \ge 0 \}
        \\ &= \min\{ \max\{ \Tr(W^\dag (X+Q)) : W \in \cornC \} : X+Q \in \spaceD, Q \ge 0 \}
        \\ &= \min\{ \max\{ \Tr(W^\dag Y) : W \in \cornC \} : Y \in \spaceD, Y \ge X \}.
        \label{eq:W_E_dual}
    \end{align}
    Since $\cornC$ has non-empty relative interior by \cref{thm:corn_relint},~\eqref{eq:W_E_primal}
    is strictly feasible and Slater's condition applies, giving $d^* = p^*$.
    But $p^* \le 1$ so $d^* \le 1$ and there is a feasible solution to
    \eqref{eq:W_E_dual} with $\max\{ \Tr(W^\dag Y) : W \in \cornC \} \le 1$.
    In other words $Y \in \cornC^\sharp$.
    Since $Y \in \cornC^\sharp \cap \spaceD$ and $Y \ge X$, we have shown that any
    $X \in \cornC^\sharp$ is dominated by an element of $\cornC^\sharp \cap \spaceD$.
\end{proof}

We are now in position to leverage \cref{thm:thin_diag} into a statement about how the structure
of the vertex set $S_0$ affects $\dsw(S, W)$.

\begin{theorem}
    Let $S$ be an $S_0$-graph and $\Sc = S^\perp + S_0$ its complement.
    Take $\blockscale$ from \cref{def:blockutils}.
    Then
    \begin{align}
        \dswbody(S) &= \{ W \ge 0 : \blockscale(W) \in \dswbodyS(\Sc) \}
        \label{eq:dswbody_psi}
        \\
        \dswbodyS(S) &= \her( \dswbodyS(S) \cap S_0' )
        \label{eq:dswbody_her_diag}
    \end{align}
    where $S_0'$ is the commutant of $S_0$.
    Equivalently, for any $W \in \psd{\hA}$,
    \begin{align}
        \dsw(S, W) &= \max\{ \Tr(W Z) : Z \ge 0, \dsw(\Sc, \blockscale(Z)) \le 1 \}
        \label{eq:max_WZ}
        \\ &= \min\{ \dsw(S, X) : X \ge W, X \in S_0' \}.
        \label{eq:min_X_diag_ge}
    \end{align}
    The maximal elements of $\dswbodyS(S)$
    are all in $S_0'$.
    In particular, for classical graphs the maximal elements of
    $\dswbodyS(S)$ are diagonal.
    \label{thm:dswSW}
\end{theorem}
\begin{proof}
    Let $n = \dim(\hA)$.
    \Cref{thm:holder_dsw} gives
    \begin{align}
        \dsw(S, W)
        &= n \max\{ \Tr(W Z) : Z \ge 0, \dsw(S^\perp + \mathbb{C}I, Z) \le 1 \}
        \\ &= \max\{ \Tr(W Z) : Z \ge 0, \dsw(S^\perp + \mathbb{C}I, n^{-1} Z) \le 1 \}
    \end{align}
    By \cref{thm:thin_diag},
    \begin{align}
        \dsw(S^\perp + \mathbb{C}I, n^{-1} Z)
        &=
        \dsw(\Sc, \blockscale(Z)).
    \end{align}
    So relation~\eqref{eq:max_WZ} is proved.

    Consider the convex corner
    \begin{align}
        \dswbodyS(S)
        &= \{ W \in \psd{\hA} : \dsw(S, W) \le 1 \}
        \label{eq:C_phi_1}
        \\ &= \{ W \in \psd{\hA} : \Tr(W Z) \le 1 \textrm{ for all } Z \ge 0, \dsw(\Sc, \blockscale(Z)) \le 1 \}
        \label{eq:C_phi_2}
        \\ &= \{ Z \in \psd{\hA} : \dsw(\Sc, \blockscale(Z)) \le 1 \}^\sharp
        \label{eq:C_phi_3}
        \\ &= \{ Z \in \psd{\hA} : \blockscale(Z) \in \dswbodyS(\Sc) \}^\sharp
    \end{align}
    where~\eqref{eq:C_phi_1} comes from \cref{thm:dswbody_abl},~\eqref{eq:C_phi_2}
    follows from~\eqref{eq:max_WZ}, and~\eqref{eq:C_phi_3} from the definition
    of anti-blocker.
    Applying the second anti-blocker theorem (\cref{thm:second_antiblocker})
    to the last line then gives~\eqref{eq:dswbody_psi}.
    This requires the r.h.s.\ to be a convex corner.  Convexity holds because $\blockscale$ is
    linear and $\dswbodyS(\Sc)$ is convex.  Hereditarity holds because $\blockscale$ is completely
    positive, giving $Z' \le Z \implies \blockscale(Z') \le \blockscale(Z)$.

    By \cref{thm:blockscale_nullspace}, $S_0'^{\perp}$ is in the null space of $\blockscale$,
    so considering~\eqref{eq:dswbody_psi} it is clear $\dswbody(S)$ meets the conditions
    for \cref{thm:antiblocker_freespace} with $\spaceD = S_0$,
    giving~\eqref{eq:dswbody_her_diag}.
    So the maximal elements of $\dswbodyS(S)$ are all in $S_0'$.
    From this and the linearity of $\dsw$,~\eqref{eq:min_X_diag_ge} follows.

    %
\end{proof}

From~\eqref{eq:min_X_diag_ge} we see that for classical graphs we gain no new
information by using non-diagonal weights: $\dsw(S, W)$ is a function of
$\dsw(S, X)$ with $X$ diagonal, which by \cref{thm:match_classical} corresponds
to the classical quantity $\lov(G, \diag(X))$.
This leads to a correspondence between the theta bodies of classical and non-commutative graphs.

\begin{theorem}
    Let $G$ be a classical graph and $S = \{ \ket{i}\bra{j} : i \simeqG j \}$.  Then
    \begin{align}
        \dswbodyS(S) &= \her \{ \diag(w) : w \in \thbody^\flat(G) \}
        \\
        \thbody^\flat(G) &= \{ \diag(W) : W \in \dswbodyS(S) \}
    \end{align}
    where $\diag(w)$ is the diagonal matrix with entries from the vector $w$ and
    $\diag(W)$ is the vector with entries taked from the diagonal of $W$.
    \label{thm:dswbodyS_classical}
\end{theorem}
\begin{proof}
    Let $\diags$ be the space of diagonal matrices.
    Starting from~\eqref{eq:dswbody_her_diag} and using the correspondence between
    $\dsw$ and $\lov$ for classical graphs and diagonal weights,
    \begin{align}
        \dswbodyS(S) &= \her( \dswbodyS(S) \cap \diags )
        \\ &= \her \{ W \in \diags : W \ge 0, \dsw(S, W) \le 1 \}
        \\ &= \her \{ \diag(w) : w \ge 0, \lov(G, w) \le 1 \}
        \\ &= \her \{ \diag(w) : w \in \thbody^\flat(G) \}.
    \end{align}
    And
    \begin{align}
        \thbody^\flat(G) &= \{ w : w \ge 0, \lov(G, w) \le 1 \}
        \\ &= \{ \diag(W) : W \ge 0, W \in \diags, \dsw(S, W) \le 1 \}
        \\ &= \{ \diag(W) : W \in \diags \cap \dswbodyS(S) \}
        \\ &= \{ \diag(W) : W \in \dswbodyS(S) \}
    \end{align}
    where the last equality follows from the fact that by \cref{thm:dswSW} the
    maximal elements of $\dswbodyS(S)$ are all in $\diags$.
\end{proof}

Applying~\eqref{eq:max_WZ} with $S=S_0$, so that $\Sc = \linop{\hA}$, gives the following corollary.

\begin{corollary}
    Let $S_0 \subseteq \linop{\hA}$ be a $C^*$-algebra and $W \in \psd{\hA}$.
    Let $\diags$ be the set of diagonal matrices.
    Then
    \begin{align}
        \dsw(S_0, W) &= \max\{ \Tr(W Z) : Z \ge 0, \opnorm{\blockscale(Z)} \le 1 \}
        \\
        \dsw(\diags, W) &= \max\{ \Tr(W Z) : Z \ge 0,
            Z_{ii} \le 1 \textnormal{ for } i \in \{1, \ldots, n\} \}
    \end{align}
    The second line is equivalent to $\gamma_2^*(W)$, the dual of the
    factorization norm from~\cite{Linial2007} (but in our case restricted to positive
    semidefinite operators).
    \label{thm:dswS0}
\end{corollary}

We now present a duality relation that generalizes both the classical
inequality~\eqref{eq:classical-ineq} as well as the duality relation from the previous chapter,
\cref{thm:holder_dsw}.
This will make clear that the extra factor of $n$ in \cref{thm:holder_dsw} was due to use of the
graph complement $S^\perp + \mathbb{C}I$: for $S_0 = \mathbb{C}I$ the scaling factor in the below
theorem takes the form $D = n I$.
When $S_0 = \diags$, this theorem reduces to the classical $\lov$ duality
relation~\eqref{eq:classical-ineq}.
In this case $S_0' = S_0 = \diags$, so equality is achievable for diagonal weights.

\begin{theorem}
    Let $S$ be an $S_0$-graph, and take $\diagscale$ from \cref{def:blockutils}.
    Then for any $W, V \in \psd{\hA}$,
    \begin{align}
        \dsw(S, V) \dsw(S^\perp + S_0, W) &\ge \Tr(V \sqrt{\diagscale} W \sqrt{\diagscale}).
        \label{eq:DWD_ineq1}
    \end{align}
    And if one of $V \in S_0'$ or $W \in S_0'$, the other can be chosen (also in $S_0'$) to achieve equality.
    \label{thm:DWD}
\end{theorem}
\begin{proof}
    Define $\Sc = S^\perp + S_0$.
    For now assume $V \in S_0'$.
    Taking~\eqref{eq:max_WZ} from \cref{thm:dswSW},
    \begin{align}
        \dsw(S, V) &=
        \max\{ \Tr(V Z) : \blockscale(Z) \in \dswbodyS(\Sc) \}
        \\ &= \max\{ \Tr(V \sqrt{D} W \sqrt{D}) : \blockscale(\sqrt{D} W \sqrt{D}) \in \dswbodyS(\Sc) \}
        \\ &= \max\{ \Tr(V \sqrt{D} W \sqrt{D}) : \commproj(W) \in \dswbodyS(\Sc) \}
        \\ &= \max\{ \Tr(V \sqrt{D} \commproj(W) \sqrt{D}) : \commproj(W) \in \dswbodyS(\Sc) \}
    \end{align}
    where the last equality uses the fact that $V \in S_0' \implies \sqrt{D} V \sqrt{D} \in S_0'$,
    so the trace can only see the $S_0'$ projection of $W$.
    Since all instances of $W$ are projected onto $S_0'$, we can equivalently just restrict this
    variable to $S_0'$, giving
    \begin{align}
        \dsw(S, V)
        &= \max\{ \Tr(V \sqrt{D} W \sqrt{D}) : W \in S_0', W \in \dswbodyS(\Sc) \}
    \end{align}
    Considering that $W \in \dswbodyS(\Sc) \iff W \ge 0, \dsw(\Sc, W) \le 1$, by linearity of $\dsw$ this can be written
    \begin{align}
        \dsw(S, V)
        &= \max\{ \Tr(V \sqrt{D} W \sqrt{D}) / \dsw(\Sc, W) : W \ge 0, W \in S_0' \}.
    \end{align}
    Therefore, under our assumption $V \in S_0'$,~\eqref{eq:DWD_ineq1} holds for all
    $W \ge 0, W \in S_0'$ and there is some $W \ge 0, W \in S_0'$ achieving equality.

    Now drop the $V, W \in S_0'$ assumptions.
    By~\eqref{eq:min_X_diag_ge} of \cref{thm:dswSW} there are $V', W' \in S_0'$
    such that $V' \ge V$, $W' \ge W$, $\dsw(S, V') = \dsw(S, V)$ and
    $\dsw(\Sc, W') = \dsw(\Sc, W)$.
    Then
    \begin{align}
        \dsw(S, V) \dsw(\Sc, W) &= \dsw(S, V') \dsw(\Sc, W')
        \\ &\ge \Tr(V' \sqrt{D} W' \sqrt{D})
        \\ &\ge \Tr(V \sqrt{D} W \sqrt{D}).
    \end{align}
\end{proof}

At this point it will be instructive to explore this bound through some simple examples.

\begin{example}
    Let $S = S_0 = \mathbb{C}I \subseteq \linop{\hA}$.  Then $\Sc = S^\perp + S_0 = \linop{\hA}$ and
    $D = nI$.
    We have $\dsw(S, V) = n \Tr(V)$ and $\dsw(\Sc, W) = \opnorm{W}$.
    \Cref{thm:DWD} then says $n \Tr(V) \opnorm{W} \ge \Tr(V \sqrt{n} W \sqrt{n}) = n \Tr(V W)$.
    And for every $V \in S_0' = \linop{\hA}$ there is $W \in S_0'$ giving equality.
\end{example}

\begin{example}
    Let $S = S_0 = \linop{\hA}$.  Then $\Sc = S^\perp + S_0 = \linop{\hA}$ and $D = n^{-1} I$.
    We have $\dsw(S, V) = \opnorm{V}$ and $\dsw(\Sc, W) = \opnorm{W}$.
    \Cref{thm:DWD} gives $\opnorm{V} \opnorm{W} \ge \Tr(V n^{-1/2} W n^{-1/2}) = n^{-1} \Tr(V W)$.
    And for every $V \in S_0' = \mathbb{C}I$ there is $W \in S_0'$ giving equality,
    i.e., $1 = \opnorm{I}^2 = n^{-1} \Tr(I)$.
\end{example}

\begin{example}
    Let $S = S_0$ for some arbitrary $C^*$-algebra $S_0 \subseteq \linop{\hA}$.
    Then $\Sc = S^\perp + S_0 = \linop{\hA}$ and $\dsw(\Sc, W) = \opnorm{W}$.
    \Cref{thm:DWD} gives $\dsw(S, V) \opnorm{W} \ge \Tr(V \sqrt{D} W \sqrt{D})$.
    If $V \in S_0'$ then $\dsw(S, V) = \max\{\Tr(V \sqrt{D} W \sqrt{D}) / \opnorm{W} \} = \Tr(V D)$,
    since the max is obtained with $W=I$.

    In the case of $S_0$ being the diagonal matrices, we have $S_0' = S_0$, $D=I$ and
    $\dsw(S, V) = \Tr(V)$.
    In this case $S$ is the classical empty graph.
\end{example}

Having proved a duality relation for $\dsw$ on an $S_0$ graph versus its $S_0$-complement, we now
translate it into a statement about the corresponding theta bodies.
Whereas classical graphs satisfy $\thbody(G) = \thbody^\flat(\Gc)$ (with $\flat$ being the classical
anti-blocker), for the case of non-commutative graphs we need a special type of anti-blocker that
bakes in the $D$ scaling seen in \cref{thm:DWD}.

Henceforth we will be using the notation
$\blockscale(\cornC) = \{ \blockscale(X) : X \in \cornC \}$
for any convex corner $\cornC$.
Note that since $\blockscale$ is convex (in fact, linear) and completely positive,
$\blockscale(\cornC)$ is a convex corner whenever $\cornC$ is.
We similarly extend $\commproj$ to operate on convex corners.

\begin{definition}
    For a convex corner $\cornC$, define
    \begin{align}
        \cornC^\abpsi &= \{ W \ge 0 : \blockscale(W) \in \cornC \}^\sharp
    \end{align}
    \label{def:abpsi}
\end{definition}

The following two lemmas will be needed for proving some useful properties of $\cornC^\abpsi$.

\begin{lemma}
    Let $\cornC$ be a convex corner such that all maximal elements are in $S_0'$.  Then
    \begin{align}
        \commproj(\cornC) &= \cornC \cap S_0'
    \end{align}
    \label{thm:corn_commproj_cap}
\end{lemma}
\begin{proof}
    Clearly
    $\commproj(\cornC) \supseteq \cornC \cap S_0'$
    because $\commproj(W) = W$ whenever $W \in S_0'$.

    Suppose $W \in \commproj(\cornC)$.  Then there is $X \in \cornC$ with
    $W = \commproj(X)$.  Since maximal elements of $\cornC$ are in $S_0'$, there is
    some $Y \in \cornC \cap S_0'$ such that $Y \ge X$.
    But $\commproj$ is completely positive, so $W = \commproj(X) \le \commproj(Y) = Y$.
    Because $\cornC$ is hereditary, $W \in \cornC$.
    Being in the image of $\commproj$, we have $W \in S_0'$.
    Therefore $\commproj(\cornC) \subseteq \cornC \cap S_0'$.
\end{proof}

\begin{lemma}
    Let $\cornC \subseteq \linop{\hA}$ be a convex corner and $Z \in \linop{\hA}$ invertible.
    Then $(Z \cornC Z^\dag)^\sharp = Z^{-\dag} \cornC^\sharp Z^{-1}$.
    \label{thm:antiblocker_mult}
\end{lemma}
\begin{proof}
    \begin{align}
        (Z \cornC Z^\dag)^\sharp
        &= \{ W \ge 0 : \Tr(W^\dag X) \le 1 \textrm{ for all } X \in Z \cornC Z^\dag \}
        \\ &= \{ W \ge 0 : \Tr(W^\dag Z X Z^\dag) \le 1 \textrm{ for all } X \in \cornC \}
        \\ &= \{ W \ge 0 : \Tr((Z^\dag W Z)^\dag X) \le 1 \textrm{ for all } X \in \cornC \}
        \\ &= \{ W \ge 0 : Z^\dag W Z  \in \cornC^\sharp \}
        \\ &= Z^{-\dag} \cornC^\sharp Z^{-1}.
    \end{align}
\end{proof}

\begin{lemma}
    For a convex corner $\cornC$, the following hold.
    \begin{thmlist}
        \item The maximal elements of $\cornC^\abpsi$ are in $S_0'$.
            \label[theorem]{thm:abpsi_max_S0}
        \item $\cornC^\abpsi = \her \blockscale(\cornC^\sharp)$, if the maximal elements of
            $\cornC$ are all in $S_0'$.
            \label[theorem]{thm:abpsi_alternate}
        \item $\cornC^{\abpsi\abpsi} = \her( \cornC \cap S_0' )$.
            \label[theorem]{thm:abpsi_involution}
    \end{thmlist}
\end{lemma}
\begin{proof}
    \itemref{thm:abpsi_max_S0}:
    Since $S_0'^\perp$ is in the null space of $\blockscale$, \cref{thm:antiblocker_freespace}
    applies.

    \itemref{thm:abpsi_alternate}:
    Let $\cornB = \{ W \ge 0 : \blockscale(W) \in \cornC \}$ so that
    $\cornC^\abpsi = \cornB^\sharp$.
    Then
    \begin{align}
        \commproj(\cornB)
        &= \{ \commproj(W) : W \ge 0, \blockscale(W) \in \cornC \}
        \\ &= \{ \commproj(W) : W \ge 0, D^{-1/2} \commproj(W) D^{-1/2} \in \cornC \}
        \\ &= \{ \commproj(W) : W \ge 0, \commproj(W) \in \sqrt{D} \cornC \sqrt{D} \}
        \\ &= \sqrt{D} \cornC \sqrt{D} \cap S_0'
        \\ &= \sqrt{D} (\cornC \cap S_0') \sqrt{D}
        \label{eq:deltaB_DCD}
    \end{align}
    where the second line uses \cref{thm:blockscale_commutes,thm:comproj_blockscale}
    and the fourth line uses that the image of $W \ge 0$ under $\commproj$ is
    $\psd{\hA} \cap S_0'$.
    Since the maximal elements of $\cornC^\abpsi$ are all in $S_0'$, and the rest follows from
    the hereditary condition, we need only consider $\cornC^\abpsi \cap S_0'$.
    \begin{align}
        \cornC^\abpsi \cap S_0'
        &= \{ X \in S_0' : X \ge 0, \Tr(X W) \le 1 \textrm{ for all } W \in \cornB \}
        \\ &= \{ X \in S_0' : X \ge 0, \Tr(\commproj(X) W) \le 1 \textrm{ for all } W \in \cornB \}
        \\ &= \{ X \in S_0' : X \ge 0, \Tr(X \commproj(W)) \le 1 \textrm{ for all } W \in \cornB \}
    \end{align}
    where the second line uses $\commproj(X) = X$ when $X \in S_0'$ and the third line follows
    because $\commproj$ is a projection (onto $S_0'$) and hence is self-adjoint (as a superoperator).
    Now, $\{ \commproj(W) : W \in \cornB\}$ is equivalent to $\{ W : W \in \commproj(\mathcal{B}) \}$, so
    \begin{align}
        \cornC^\abpsi \cap S_0'
        &= \{ X \in S_0' : X \ge 0, \Tr(X W) \le 1 \textrm{ for all } W \in \commproj(\cornB) \}
        \\ &= S_0' \cap (\commproj(B))^\sharp.
    \end{align}
    Since $S_0'^\perp$ is in the null space of $\commproj$, \cref{thm:antiblocker_freespace} applies
    and the maximal elements of $(\commproj(B))^\sharp$ are all in $S_0'$.
    Apply \cref{thm:corn_commproj_cap} to the right hand side to give
    $\cornC^\abpsi \cap S_0' = \commproj((\commproj(B))^\sharp)$.
    Substituting~\eqref{eq:deltaB_DCD} gives
    \begin{align}
        \cornC^\abpsi \cap S_0'
        &= \commproj((\sqrt{D}(\cornC \cap S_0') \sqrt{D})^\sharp)
        \\ &= \commproj(D^{-1/2} (\cornC \cap S_0')^\sharp D^{-1/2})
        \\ &= \commproj(D^{-1/2} \cornC^\sharp D^{-1/2})
        \\ &= \blockscale(C^\sharp)
    \end{align}
    where the second line uses \cref{thm:antiblocker_mult}, and the third line uses
    $\cornC^\sharp = (\cornC \cap S_0')^\sharp$ because the maximal elements of $\cornC$ are in
    $S_0'$ and the anti-blocker only cares about the maximal elements.
    The last line uses \cref{thm:comproj_blockscale}.

    Since, by~\itemref{thm:abpsi_max_S0}, all maximal elements of $\cornC^\abpsi$ are in $S_0'$, we have
    \begin{align}
        \cornC^\abpsi
        &= \her(\cornC^\abpsi \cap S_0')
        \\ &= \her(\blockscale(C^\sharp))
    \end{align}

    \itemref{thm:abpsi_involution}:
    The maximal elements of $\cornC^\abpsi$ are in $S_0'$ so we can apply~\itemref{thm:abpsi_alternate}.
    \begin{align}
        \cornC^{\abpsi\abpsi}
        &= \her( \blockscale(\cornC^{\abpsi\sharp}))
        \\ &= \her( \blockscale(\{ W \ge 0 : \blockscale(W) \in \cornC \}^{\sharp\sharp}))
        \\ &= \her( \blockscale(\{ W \ge 0 : \blockscale(W) \in \cornC \}))
    \end{align}
    where the second line uses \cref{def:abpsi} and the third line uses the second anti-blocker
    theorem.  That requires the argument to be a convex corner, which it is: convexity follows from
    linearity of $\blockscale$ and convexity of $\cornC$ whereas hereditarity follows from
    hereditarity of $\cornC$ and $\blockscale$ being positive semidefinite.
    Moving the first $\blockscale$ into the set notation, we have
    \begin{align}
        \cornC^{\abpsi\abpsi}
        &= \her( \{ \blockscale(W) : W \ge 0, \blockscale(W) \in \cornC \})
        \\ &= \her( \cornC \cap S_0' )
    \end{align}
    where the second line holds because the image of $\blockscale$ is $S_0'$.
\end{proof}

Finally, we apply the $\abpsi$ anti-blocker to give a duality relation between theta bodies.
When $S_0 = \diags$ (the diagonal matrices), this reproduces the classical result
$\thbody(\Gc) = \thbody^\flat(G)$
When $S_0 = \mathbb{C}I$ this reproduces the theta body relation from \cref{thm:holder_dsw}.

\begin{theorem}
    Let $S$ be an $S_0$-graph and $\Sc = S^\perp + S_0$ its complement graph.  Then
    \begin{align}
        \dswbodyS(\Sc) &= \dswbody^{\sharp\abpsi}(S)
    \end{align}
    \label{thm:dswbody_abpsi}
\end{theorem}
\begin{proof}
    Starting from \eqref{eq:dswbody_psi},
    \begin{align}
        \dswbody(S) &= \{ Z \ge 0 : \blockscale(Z) \in \dswbodyS(\Sc) \}
        \\ \dswbodyS(S) &= \{ Z \ge 0 : \blockscale(Z) \in \dswbodyS(\Sc) \}^\sharp
        \\ &= \dswbody^{\sharp\abpsi}(\Sc)
    \end{align}
    Applying this to $\Sc$, which is also an $S_0$ graph, gives
    $\dswbodyS(\Sc) = \dswbody^{\sharp\abpsi}(S^{cc}) = \dswbody^{\sharp\abpsi}(S)$.
\end{proof}


\section{A sandwich theorem and perfect graphs}
\label{sec:facets}

For classical graphs the theta body is sandwiched between the vertex packing polytope $\vp(G)$,
defined as the convex hull of indicator functions of independent sets, and the fractional
vertex packing polytope, equal to $\vp^\flat(\Gc)$:~\cite{Grtschel1986}
\begin{align}
    \vp(G) \subseteq \thbody(G) \subseteq \vp^\flat(\Gc).
    \label{eq:classical_sandwich}
\end{align}
Or, considering the complement graph and making use of the fact that
$\thbody(\Gc) = \thbody^\flat(G)$,
\begin{align}
    \vp(\Gc) \subseteq \thbody^\flat(G) \subseteq \vp^\flat(G).
    \label{eq:classical_sandwich_cmpl}
\end{align}
We will generalize this latter statement to non-commutative graphs.
Note that $\vp(\Gc)$ is the polytope of indicator functions of cliques of $G$.
Given that
\begin{align}
    \alpha(G)      &= \max\left\{ \sum w_i : w \in \vp(G) \right\}
    \\ \lov(G)     &= \max\left\{ \sum w_i : w \in \thbody(G) \right\}
    \\ \chi^*(\Gc) &= \max\left\{ \sum w_i : w \in \vp^\flat(\Gc) \right\}
\end{align}
(where $\alpha$ is independence number and $\chi^*$ is fractional chromatic
number),~\eqref{eq:classical_sandwich} can be seen as a more granular version of the well known
sandwich inequality $\alpha(G) \le \lov(G) \le \chi^*(\Gc)$.

Also from~\cite{Grtschel1986} comes the marvelous result that the following are equivalent:
\begin{enumerate}
    \item $G$ is a perfect graph.
    \item $\thbody(G)$ is a polytope.
    \item $\thbody(G) = \vp(G)$.
    \item $\thbody(G) = \vp^\flat(\Gc)$.
\end{enumerate}
We will show this partially generalizes to non-commutative graphs.
There is currently no definition of ``perfect'' for non-commutative graphs, so we may take the
non-commutative generalization of $\vp(G) = \vp^\flat(\Gc)$ to be the definition.
Since we are using an operator generalization of convex corners and anti-blockers, there are
(at least) two possible analogues of ``polytope'': we could consider convex corners
generated by finitely many vertices, or those bounded by finitely many inequalities.
We will find that if $\dswbodyS(S)$ is finitely generated then it is equal to the non-commutative
generalization of $\vp(\Gc)$, but that the converse is not true.
This can be considered a first dipping of the toes into the theory of non-commutative perfect
graphs, with many questions remaining open.

The following definition from~\cite{btw2019} gives the non-commutative generalization of
$\vp(\Gc)$ that we will be using.
They also define $S$-abelian and $S$-clique projectors, the former being analogous to independent
sets and the latter being a looser definition of clique, but we will not be using those here.

\begin{definition}
    Let $S \subseteq \linop{\hA}$ be a non-commutative graph.  A projector $P \in \linop{\hA}$
    is called $S$-full if $P \linop{\hA} P \subseteq S$.
    We define the convex corner generated by these projections,
    \begin{align}
        \fp(S) = \her( \convcl\{ P : P \textnormal{ an $S$-full projection} \} )
    \end{align}
\end{definition}

\begin{theorem}
    Let $S \subseteq \linop{\hA}$ be an $S_0$-graph.  If $P$ is a maximal $S$-full projection then
    $P \in S_0'$.  Consequently, the maximal elements of $\fp(S)$ are in $S_0'$.
    \label{thm:fp_in_S1}
\end{theorem}
\begin{proof}
    Let $P$ be a maximal $S$-full projection.  Then $P = \sum_{i=1}^m \ket{\psi_i}\bra{\psi_i}$ for some
    collection of normalized vectors $\{ \ket{\psi_i} \}$, and $\ket{\psi_i}\bra{\psi_j} \in S$ for
    all $i,j \in \{ 1, \dots, m \}$.
    By the definition of an $S_0$ graph, $S_0 S = S S_0 = S$.
    Therefore, for any $X \in S_0$ we have $X \ket{\psi_i}\bra{\psi_j} \in S$ and
    $\ket{\psi_i}\bra{\psi_j} X^\dag \in S$.
    It must be that $X \ket{\psi_i} \in \linspan\{ \ket{\psi_1}, \dots, \ket{\psi_m} \}$ for all $X \in S_0$ and
    all $i$.  Otherwise, $\linspan\{ X \ket{\psi_i}, \ket{\psi_1}, \dots, \ket{\psi_m} \}$ would
    define a larger $S$-full space, in contradiction to $P$ being maximal.

    Consider a unitary $U \in S_0$.  By the above reasoning, $U \ket{\psi_i}$ is in the support of
    $P$ for all $i \in \{ 1, \dots, m \}$.
    Then $U P U^\dag = \sum_i U \ket{\psi_i}\bra{\psi_i} U^\dag$ has support contained in the
    support of $P$.  Since $U P U^\dag$ is also a projector, it must be that $U P U^\dag = P$.
    Equivalently, $U P = P U$.
    Since $P$ commutes with all unitary $U \in S_0$, it must also commute with the algebra generated
    by those unitaries, which is all of $S_0$.  Therefore $P \in S_0'$.

    Since each maximal $S$-full projection is in $S_0'$, so must be the closure of their convex
    hull,
    \begin{align}
        \convcl\{ P : P \textnormal{ a maximal $S$-full projection} \} \subseteq S_0'.
    \end{align}
    Every maximal element of $\fp(S)$ must be a member of this, because every $S$-full projection
    is dominated by some maximal $S$-full projection.
\end{proof}

We are now ready to present our sandwich theorem, the generalization
of~\eqref{eq:classical_sandwich_cmpl} to $S_0$-graphs.

\begin{theorem}
    \begin{align}
        \fp(S) &\subseteq \dswbodyS(S) \subseteq \fp^\abpsi(\Sc)
        \label{eq:dswbody_sandwich}
    \end{align}
    \label{thm:dswbody_sandwich}
\end{theorem}
\begin{proof}
    Let $P$ be an $S$-full projector.
    Since $Y = \proj{P}$ is feasible for~\eqref{eq:weighted_dsw_min_Y}
    with $\lambda=1$, we have $\dsw(S, P) \le 1$ so $P \in \dswbodyS(S)$.
    Since $\dswbodyS(S)$ is a convex corner and contains all $S$-full projectors, it must also
    contain the convex corner generated by the $S$-full projectors.
    Therefore $\fp(S) \subseteq \dswbodyS(S)$.

    Applying the above results to the graph $\Sc$, we have
    \begin{align}
        \fp(\Sc) &\subseteq \dswbodyS(\Sc)
        \\
        \{ W \ge 0 : \blockscale(W) \in \fp(\Sc) \} &\subseteq
            \{ W \ge 0 : \blockscale(W) \in \dswbodyS(\Sc) \}
        \\
        \{ W \ge 0 : \blockscale(W) \in \fp(\Sc) \}^\sharp &\supseteq
            \{ W \ge 0 : \blockscale(W) \in \dswbodyS(\Sc) \}^\sharp
        \\
        \fp^\abpsi(\Sc) &\supseteq \dswbody^{\sharp\abpsi}(\Sc)
        \\ &= \dswbodyS(S)
    \end{align}
    where the third line uses the fact that taking the anti-blocker of both sides reverses the
    inclusion order (\cref{thm:antiblocker_inclusion}) and the last uses the duality relation
    of \cref{thm:dswbody_abpsi}.
\end{proof}

It's worth noting this sandwich theorem is symmetric under graph complement.
Taking the~$\abpsi$~anti-blocker of~\eqref{eq:dswbody_sandwich} and considering that this reverses
inclusion order and that $\dswbody^{\sharp\abpsi}(S) = \dswbodyS(\Sc)$, we get
$\fp^\abpsi(S) \supseteq \dswbodyS(\Sc) \supseteq \fp(\Sc)$.
This is equivalent to~\eqref{eq:dswbody_sandwich} applied to the complement graph.

It is also worth noting that the right inclusion of~\eqref{eq:dswbody_sandwich} was derived from
the left inclusion without using any special properties of $\fp(S)$.
Any convex corner contained in $\dswbodyS(S)$ would give a similar sandwich theorem.
This opens the possibility of tightening~\eqref{eq:dswbody_sandwich} by finding some convex corner
larger than $\fp(S)$ but still contained in $\dswbodyS(S)$.

Since a classical graph is perfect if and only if $\vp(G) = \vp^\flat(\Gc)$,
we take the non-commutative analogue to be one possible definition of perfect for non-commutative
graphs.

\begin{definition}
    An $S_0$-graph $S$ is \emph{$\fp$-perfect} if $\fp(S) = \fp^\abpsi(\Sc)$.
\end{definition}

Note that, by~\cref{thm:dswbody_sandwich}, $\fp$-perfect graphs necessarily also satisfy
$\dswbodyS(S) = \fp(S)$.

\begin{theorem}
    For a graph $G$ define $S = \linspan\{ \ket{i}\bra{j} : i \simeqG j \}$
    and $S_0 = \linspan\{ \ket{i}\bra{i} : i \in V(G) \}$.
    Then $S$ is $\fp$-perfect if and only if $G$ is perfect.
\end{theorem}
\begin{proof}
    The graph $G$ is perfect if and only if $\vp(\Gc) = \vp^\flat(G)$
    where $\vp$ is the vertex packing polytope, the convex hull of incidence vectors (along the
    diagonal) of independent sets and $\vp^\flat(G) = \vp^\sharp(G) \cap S_0'$ is the diagonal
    anti-blocker~\cite{Grtschel1986}.
    Note $\vp(\Gc)$ is the convex hull of incidence vectors of cliques,
    \begin{align}
        \vp(\Gc) = \conv\left\{ \sum_{i \in X} \ket{i}\bra{i} : X \subseteq V(G)
            \textrm{ is a clique of } G \right\}.
    \end{align}

    By~\cite[theorem 3.5, corollary 3.7]{btw2019},
    \begin{align}
    \commproj(\fp(S)) &= \fp(S) \cap S_0' = \vp(\Gc)
    \label{eq:fpS_vpGc}
    \\
    \commproj(\fp^\sharp(S)) &= \fp^\sharp(S) \cap S_0' = \vp^\flat(\Gc)
    \label{eq:fpsS_vpfGc}
    \end{align}

    And by~\cite[lemma 3.6]{btw2019}, if $\cornA$ is a diagonal convex corner and $\cornB$ is a
    convex corner such that $\cornA = \diags \cap \cornB = \commproj(\cornB)$,
    where $\diags = S_0 = S_0'$ are the diagonal operators, then
    $\cornA^\flat = \diags \cap \cornB^\sharp = \commproj(\cornB)$.
    Taking $\cornB = \{ W \ge 0 : \commproj(W) \in \fp(\Sc) \}$, we have
    $\diags \cap \cornB = \commproj(\cornB) = \diags \cap \fp(\Sc)$.
    By~\eqref{eq:fpS_vpGc}, $\diags \cap \fp(\Sc) = \vp(G)$, so taking $\cornA = \vp(G)$
    gives $\cornA^\flat = \vp^\flat(G) = \commproj(\cornB^\sharp)$.
    By \cref{def:abpsi}, $\cornB^\sharp = \fp^\abpsi(\Sc)$.
    So
    \begin{align}
        \commproj(\fp^\abpsi(\Sc)) &= \vp^\flat(G).
        \label{eq:dfpabpsi_vpflat}
    \end{align}

    Suppose $S$ is $\fp$-perfect.  Then
    \begin{align}
        \fp(S) = \fp^\abpsi(\Sc)
        &\implies \commproj(\fp(S)) = \commproj(\fp^\abpsi(\Sc))
        \\ &\implies \vp(\Gc) = \vp^\flat(G),
    \end{align}
    where the second implication follows from~\eqref{eq:fpS_vpGc} and~\eqref{eq:dfpabpsi_vpflat}.
    So $G$ is perfect.

    On the other hand, suppose $G$ is perfect so $\vp(\Gc) = \vp^\flat(G)$.
    Then
    \begin{align}
        \vp(\Gc) = \vp^\flat(G)
        &\implies \fp(S) \cap S_0' = \commproj(\fp^\sharp(\Sc))
        \\ &\implies \fp(S) \supseteq \commproj(\fp^\sharp(\Sc))
        \\ &\implies \her(\fp(S)) \supseteq \her(\commproj(\fp^\sharp(\Sc)))
        \\ &\implies \fp(S) \supseteq \fp^\abpsi(\Sc),
    \end{align}
    where the first implication follows from~\eqref{eq:fpS_vpGc}-\eqref{eq:fpsS_vpfGc}
    and the last from $\fp(S) = \her(\fp(S))$ and \cref{thm:abpsi_alternate}
    (since $\commproj$ and $\blockscale$ are equivalent when $S_0 = \diags$).
    By \cref{thm:dswbody_sandwich}, $\fp(S) \subseteq \fp^\abpsi(\Sc)$
    so in fact $\fp(S) = \fp^\abpsi(\Sc)$ and $S$ is $\fp$-perfect.
\end{proof}

For classical graphs the weak perfect graph theorem states that a graph is perfect if and only if
its complement is perfect.
This holds also for $\fp$-perfect graphs.

\begin{theorem}
    If an $S_0$-graph $S$ is $\fp$-perfect then $\Sc$ is $\fp$-perfect.
\end{theorem}
\begin{proof}
    Suppose $S$ is $\fp$-perfect.  Starting from the definition of $\fp$-perfect and applying
    \cref{thm:abpsi_involution} we have
    \begin{align}
        \fp(S) &= \fp^\abpsi(\Sc)
        \\
        \fp(S)^\abpsi &= \fp^{\abpsi\abpsi}(\Sc)
        \\ &= \her(\fp(\Sc) \cap S_0')
        \\ &= \fp(\Sc)
    \end{align}
    where the last equality follows from the fact that $\fp(\Sc)$ is generated by
    elements of $S_0'$ (\cref{thm:fp_in_S1}).
\end{proof}

We give an example of an $\fp$-perfect graph which will also later serve as an
important counterexample.
This $S_0$-graph is not a classical graph, though its complement is the classical empty graph (or,
it would be aside from the fact we take $S_0 = \mathbb{C}I$ rather than the algebra of diagonal
matrices).

\begin{example}
    The graph
    \begin{align}
        S = \left\{
            \begin{pmatrix} a & b \\ c & a \end{pmatrix}
            : a, b, c \in \mathbb{C}
        \right\}
        \label{eq:perfect_2x2_S}
    \end{align}
    is an $\fp$-perfect $S_0$-graph with $S_0 = \mathbb{C}I$.
    \label{thm:perfect_2x2}
\end{example}
\begin{proof}
    Clearly no $S$-full projector can be rank-2 because $S \ne \linop{\hA}$.
    Rank-1 projectors have trace 1, so to be in~\eqref{eq:perfect_2x2_S} they must have $a=1/2$.
    In fact, any such projector is easily verified to be $S$-full.
    We have then
    \begin{align}
        \fp(S) &= \her \left( \convcl \left\{
            \frac{1}{2} \begin{pmatrix} 1 & \phi^\dag \\ \phi & 1 \end{pmatrix}
            : \phi \in \mathbb{C}, \abs{\phi} = 1 \right\} \right)
            \label{eq:perfect_2x2_fp}
        \\ &= \{ M \in \linop{\mathbb{C}^2} : M \ge 0, M_{11} \le 1/2, M_{22} \le 1/2 \}.
    \end{align}
    The complement graph is
    \begin{align}
        \Sc = S^\perp + \mathbb{C} I = \{ M \in \linop{\mathbb{C}^2} : M \textnormal{ is diagonal} \}.
    \end{align}
    Again there are no rank-2 $\Sc$-full projectors.  And clearly the only rank-1
    projectors in $\Sc$ are the projectors onto the two basis vectors.
    We have then
    \begin{align}
        \fp(\Sc) &= \her \left( \convcl \left\{
            \begin{pmatrix} 1 & 0 \\ 0 & 0 \end{pmatrix},
            \begin{pmatrix} 0 & 0 \\ 0 & 1 \end{pmatrix}
         \right\} \right).
    \end{align}
    Since the anti-blocker only cares about the extreme points of $\fp(\Sc)$, the anti-blocker is
    \begin{align}
        \fp^\sharp(\Sc)
        &= \{ M \in \linop{\mathbb{C}^2} : M \ge 0, M_{11} \le 1, M_{22} \le 1 \}
        \\ &= 2 \fp(S).
    \end{align}
    With $S_0 = \mathbb{C}I \subseteq \linop{\mathbb{C}^2}$
    we have $\blockscale(W) = W/2$ so
    \begin{align}
        \fp^\abpsi(\Sc)
        &= \{ W \ge 0 : \blockscale(W) \in \fp(\Sc) \}^\sharp
        \\ &= \{ W \ge 0 : W/2 \in \fp(\Sc) \}^\sharp
        \\ &= ( 2 \fp(\Sc) )^\sharp
        \\ &= \frac{1}{2} \fp^\sharp(\Sc)
        \\ &= \fp(S)
    \end{align}
\end{proof}

For classical graphs,~\cite{Grtschel1986} showed that a graph $G$ is perfect if and only if
$\thbody(G)$ is a polytope, and that facets of $\thbody(G)$ correspond to cliques of $G$.
The remainder of this section will be devoted to exploring to which extent this generalizes
to $S_0$-graphs.
We are using convex corners on operator spaces, which are more complicated than the diagonal convex
corners which $\thbody(G)$ lives in.
Diagonal convex corners are polyhedral iff they are finitely generated iff they are defined by
finitely many inequalities.
This is not true for the non-diagonal convex corners we are using.
For example, $\her(\{ I \})$, the set of bounded semidefinite operators, is a finitely generated
convex corner but is not polyhedral.
Before proceeding we will need some groundwork on the geometry of convex corners and anti-blockers.

\begin{definition}
    Let $\cornC \subseteq \herm{\hA}$ be a nonempty closed convex set and $T \subseteq \cornC$ a subset.
    A \emph{supporting hyperplane} of $\cornC$ at $T$ is an affine subspace
    $\{ X \in \herm{\hA} : \Tr(X Y) = \alpha \}$,
    with $Y \in \herm{\hA}$ and $\alpha \in \mathbb{R}$,
    such that $\Tr(X Y) \le \alpha$ for all $X \in \cornC$ and $\Tr(X Y) = \alpha$
    for all $X \in T$.

    A \emph{vertex} $X \in \cornC$ is an element on the boundary of $\cornC$ such that the intersection
    of all supporting hyperplanes of $\cornC$ at $X$ is an affine subspace of dimension 0.
    In other words, a supporting hyperplane can be ``wiggled'' a small amount in every direction,
    pivoting on $X$.

    A \emph{facet} $\mathcal{F} \subseteq \cornC$ is a subset of the boundary of $\cornC$,
    of affine codimension 1 (affine dimension $\dim(\linop{\hA})-1$),
    defined by a single supporting hyperplane:
    $\mathcal{F} = \{ X \in \cornC : \Tr(X Y) = \alpha \}$.
    Note that facets are necessarily convex.
    \label{def:vertex_facet}
\end{definition}

Many authors only require that vertices are not part of any line segment in $\cornC$.
We take the stricter definition from~\cite[definition 2.3]{gallier2008notes}.
The corners of a cube are vertices, the boundary points of a closed ball or cylinder are not.
Also, the boundary points of a closed ball are not facets: they are defined by a single supporting
hyperplane but are not affine codimension 1.

We will need the following facts regarding the geometry of convex corners.

\begin{theorem}
    Let $\cornC \subseteq \linop{\hA}$ be a convex corner.  Then the following hold.
    \begin{thmlist}
        \item
        If $\dim(\hA) > 1$ and $\cornC$ has a facet then $\cornC$ is necessarily of full dimension,
        containing some positive definite element.
        \label[theorem]{thm:facet_gives_full_dim}

        \item
        For $\dim(\hA) > 1$,
        supporting hyperplanes defining facets of $\cornC$ do not pass through the origin.
        Therefore, facets always take the form
        \begin{align}
            \mathcal{F} = \{ X \in \cornC : \Tr(X Y) = 1 \}.
            \label{eq:polar_hyperplane}
        \end{align}
        \label[theorem]{thm:facet_alpha_1}

        \item
        If~\eqref{eq:polar_hyperplane} is a facet then $Y$ is a vertex of $\cornC^\sharp$.
        \label[theorem]{thm:facet_is_vertex}

        \item
        Non-zero vertices of $\cornC$ are maximal in the sense that if $X$ is a vertex and
        $Y \ge X, Y \ne X$ then $Y \not\in \cornC$.
        \label[theorem]{thm:vertex_is_maximal}

        \item
        Suppose $\cornC$ is finitely generated and
        $\{ X_i : i \in \{ 0, \ldots, m \} \}$ is a minimal set of generators,
        so $\cornC = \her(\conv\{ X_i : i \in \{ 0, \ldots, m \} \} )$.
        Then each $X_i$ is a vertex of $\cornC$ and
        \begin{align}
            \mathcal{F} = \{ Z \in \cornC^\sharp : \Tr(X_i Z) = 1 \}
            \label{eq:facet_from_vert}
        \end{align}
        is a facet of $\cornC^\sharp$.
        \label[theorem]{thm:finitely_generated_is_facet}
    \end{thmlist}
\end{theorem}
\begin{proof}
    Proofs are in \cref{sec:corner_geometry}.
\end{proof}

Intriguingly, many of these facts are not true for diagonal convex corners.
That is to say, the usage of the Loewner order for the hereditarity condition,
i.e.\ $X \in \cornC, 0 \le Y \le X \implies Y \in \cornC$, imposes stricter requirements
on the existence of vertices and facets for non-commutative convex corners as compared
to diagonal convex corners, in which hereditarity is defined by elementwise less-than of
vectors.

For example, consider the diagonal convex corner
$\{ (x, y) \subseteq \mathbb{R}^2 : x,y \ge 0, x,y \le 1\}$.
The edge $y=0$ forms a facet whose supporting hyperplane passes through the origin,
in violation of \cref{thm:facet_alpha_1}.
The point $x=1, y=0$ is a vertex but is not maximal, in violation of \cref{thm:vertex_is_maximal}.
The analogous non-commutative convex corner is
$\cornC = \{ X \in \psd{\mathbb{R}^2} : \opnorm{X} \le 1 \}$.
Here too we have that $\diag(0,1)$ is not a maximal element, but in this case it is also not
a vertex: the tangent cone of $\cornC$ at $\diag(0,1)$ contains the line
$\left( \begin{smallmatrix} 0 & \epsilon \\ \epsilon & 1 \end{smallmatrix} \right)$.

We are ready to show that facets of $\dswbody(S)$ correspond to $S$-full projectors.
The derivation follows in spirit that of~\cite{Grtschel1986}, but using the language of compatible
matrices from~\cite[chapter 29]{knuth94}.
The following lemma states that if we take optimal solutions for $\dsw$ for a non-commutative graph
and its complement that saturate the duality relation of~\cref{thm:holder_dsw}, then the matrices
associated with the Schur complement form of the $\dsw$ SDP (\cref{thm:dsw_grotschel})
will be orthogonal to each other.
From this we can read off some useful relations among the blocks of those matrices.
When $\dswbody(S)$ has a facet, there are multiple linearly independent instances of these
relations, forcing the matrix from~\cref{thm:dsw_grotschel} into a particular form: it must
be rank-1.
Such a solution must necessarily correspond to an $S$-full projector.

\begin{lemma}
    Let $S$ be a non-commutative graph.
    Suppose $W, V \in \psd{\hA}$ saturate the inequality of \cref{thm:holder_dsw}
    so\footnote{We use the complement $S^\perp + \mathbb{C}I$ regardless of whether
    $S_0 = \mathbb{C}I$.  It is possible to make a more complicated version of this lemma using
    the complement $S^\perp + S_0$, but this will not be necessary.}
    \begin{align}
        \dsw(S, V) \dsw(S^\perp + \mathbb{C}I, W) = n \Tr(V W).
        \label{eq:compatible_hypothesis}
    \end{align}
    Let $Z$ be optimal for~\eqref{eq:dsw_grotschel} for $\lambda = \dsw(S, V)$
    and $Z'$ be optimal for $\lambda' = \dsw(S^\perp + \mathbb{C}I, W)$.
    Then
    \begin{align}
        \lambda \ket{\vw} &= n Z' \ket{\vv}
        \label{eq:WZV}
        \\
        \lambda' \ket{\vv} &= n Z \ket{\vw}
        \label{eq:VZW}
        \\
        \ket{\vv}\bra{\vw} &= n Z Z'.
        \label{eq:VWZZ}
    \end{align}
    \label{thm:compatible}
\end{lemma}
\begin{proof}
    Consider the block matrices from~\eqref{eq:dsw_grotschel} for $S$ and $S^\perp + \mathbb{C}I$,
    with the second conjugated by a diagonal matrix to change the sign of and scale the second row
    and column.
    \begin{align}
        M &=
        \left( \begin{array}{c|c}
            \lambda & \bra{\vv} \\
            \hline
            \ket{\vv} & Z
        \end{array} \right)
        \\ M' &=
        \left( \begin{array}{c|c}
            1 & 0 \\
            \hline
            0 & -n I
        \end{array} \right)
        \left( \begin{array}{c|c}
            \lambda' & \bra{\vw} \\
            \hline
            \ket{\vw} & Z'
        \end{array} \right)
        \left( \begin{array}{c|c}
            1 & 0 \\
            \hline
            0 & -n I
        \end{array} \right)
        \\ &=
        \left( \begin{array}{c|c}
            \lambda' & -n \bra{\vw} \\
            \hline
            -n \ket{\vw} & n^2 Z'
        \end{array} \right)
    \end{align}

    Their product is
    \begin{align}
        M M' &=
        \left( \begin{array}{c|c}
            \lambda \lambda' - n \braket{\vv}{\vw}
            & -\lambda n \bra{\vw} + n^2 \bra{\vv} Z'
            \\ \hline
            \lambda' \ket{\vv} - n Z \ket{\vw}
            & -n \ket{\vv}\bra{\vw} + n^2 Z Z'
        \end{array} \right)
        \label{eq:grot_prod}
    \end{align}
    We will show that $M$ and $M'$ are orthogonal under the Hilbert-Schmidt inner product,
    i.e., $\Tr(M M') = 0$.
    The upper-left block vanishes: $\braket{\vv}{\vw} = \Tr(VW)$ and
    by~\eqref{eq:compatible_hypothesis}, $\lambda \lambda' = n \Tr(VW)$.
    What remains of the trace is the lower-right block,
    \begin{align}
        \Tr(M M') &= -n \Tr( \ket{\vv}\bra{\vw} ) + n^2 \Tr(Z Z')
        \\ &= -n \Tr(VW) + n^2 \Tr(Z Z').
    \end{align}

    Since $Z \in S \ot \linop{B}$ and $Z' \in (S^\perp + \mathbb{C}I) \ot \linop{B}$,
    the Hilbert-Schmidt inner product of $Z$ and $Z'$ can only see the projection of
    these variables onto $\mathbb{C}I \ot \linop{B}$.
    So we have
    \begin{align}
        \Tr(Z Z')
        &= \Tr\left( ( n^{-1} I \ot \TrA Z ) ( n^{-1} I \ot \TrA Z' ) \right)
        \\ &= n^{-1} \Tr\left( ( \TrA Z ) ( \TrA Z' ) \right)
        \\ &= n^{-1} \Tr( V^T W^T )
        \\ &= n^{-1} \Tr( V W ),
    \end{align}
    where the third line comes from the condition $\Tr_A Z = V^T$, $\Tr_A Z' = W^T$
    of~\eqref{eq:dsw_grotschel}.
    Therefore $\Tr(M M') = 0$ and the $M$ and $M'$ operators are orthogonal.
    Positive semidefinite orthogonal operators have orthogonal supports, so in fact
    $M M' = 0$.
    The blocks of $M M'$ listed in~\eqref{eq:grot_prod} being zero
    gives~\eqref{eq:WZV}-\eqref{eq:VWZZ}.
\end{proof}

\begin{theorem}
    If $S \subseteq \linop{A}$ is a non-commutative graph and
    \begin{align}
        \mathcal{F} = \{ X \in \dswbody(S) : \Tr(X V) = 1 \}
    \end{align}
    is a facet of $\dswbody(S)$ then $V$ is a maximal $S$-full projector.
    Note that by \cref{thm:facet_alpha_1}, all facets of a convex corner take this form.
    \label{thm:facet_is_clique}
\end{theorem}
\begin{proof}
    Define $\Sc = S^\perp + \mathbb{C}I$ and let $n = \dim(\hA)$.
    By \cref{thm:facet_is_vertex}, $V$ is a vertex of $\dswbodyS(S)$.
    By \cref{thm:vertex_is_maximal} vertices are maximal, so $\dsw(S, V) = 1$.

    Being a facet, $\mathcal{F}$ has $\dim(\linop{\hA}) = n^2$ elements spanning a
    $\dim(\linop{\hA})-1 = n^2-1$
    dimensional affine subspace.
    Label these elements $X_1, \ldots, X_{n^2}$.

    Since each $X_i \in \mathcal{F}$, we have
    \begin{align}
        \Tr(X_i V) &= 1
        \\ \dsw(\Sc, n^{-1} X_i) &= 1
    \end{align}
    where the second line comes from~\eqref{eq:holder_dsw_corner} of \cref{thm:holder_dsw}:
    $X_i$ being on the boundary of $\dswbody(S)$ means $n^{-1} X_i$ is on the boundary of
    $\dswbodyS(\Sc)$.
    Defining $W_i = n^{-1} X_i$, we have
    \begin{align}
        n \Tr(W_i V) &= 1
        \\ \dsw(\Sc, W_i) &= 1.
    \end{align}
    By \cref{thm:compatible} there is
    $Z$ optimum for~\eqref{eq:dsw_grotschel} for $\dsw(S, V)$ with
    \begin{align}
        \ket{\vv} &= n Z \ket{\vw_i}.
        \label{eq:facets_VZW}
    \end{align}
    Since the $X_i$ span an $n^2-1$ dimensional affine subspace of $\linop{\hA}$,
    so do the $W_i$.
    And the $\ket{\vw_i}$ span an $n^2-1$ dimensional affine subspace of $\hA \ot \hB$.
    Their differences then span a codimension $1$ subspace of $\hA \ot \hB$.
    Since
    \begin{align}
        Z (\ket{\vw_i} - \ket{\vw_j})
        &= n^{-1} (\ket{V} - \ket{V})
        \\ &= 0,
    \end{align}
    we have that $Z$ is rank $1$.
    Considering~\eqref{eq:facets_VZW} and
    $\braket{\vv}{\vw_i} = \Tr(W_i V) = n^{-1}$,
    we have
    \begin{align}
        Z &= \proj{\vv}.
    \end{align}
    This gives $\TrA Z = (V^T)^2$ but by~\eqref{eq:dsw_grotschel} we have
    $\TrA Z = V^T$.  Therefore $V$ is a projector.
    For any $M \in \linop{\hA}$ we have
    \begin{align}
        V M V^\dag &= \TrB( (I \ot M^T) \proj{\vv} )
        \\ &= \TrB( (I \ot M^T) Z )
        \\ &\in S
    \end{align}
    where the last relation follows from $Z \in S \ot \linop{\hB}$,
    a requirement of~\eqref{eq:dsw_grotschel}.
    Therefore $V$ is an $S$-full projector.
    Since $V$ is a vertex of $\dswbodyS(S)$, it is maximal in $\dswbodyS(S)$.
    Therefore $V$ is a maximal $S$-full projector.
\end{proof}

After having shown that facets of $\thbody(G)$ correspond to cliques of $G$,~\cite{Grtschel1986}
is able to immediately state that $\thbody(G)$ being a polytope implies $G$ is perfect:
$\thbody(G)$ being defined by clique constraints forces $\thbody(G) = \vp^\flat(\Gc)$.
Since $\thbody^\flat(G) = \thbody(\Gc)$ must also be polyhedral we similarly have
$\thbody(G) = \vp(G)$, giving $\vp(G) = \vp^\flat(\Gc)$ (i.e.\ $G$ is perfect).
For non-commutative graphs the situation is not so simple.
What we will be able to show is that if $\dswbodyS(S)$ is finitely generated, then it is
equal to $\fp(S)$.
But the converse does not necessarily hold.
And $\dswbodyS(S) = \fp(S)$ does not necessarily mean that $S$ is $\fp$-perfect.

\begin{theorem}
    If $\dswbodyS(S)$ is finitely generated then it is equal to $\fp(S)$.
    \label{thm:fin_gen_dswbody_fp}
\end{theorem}
\begin{proof}
    Suppose $\dswbodyS(S)$ is finitely generated and let
    $\{ X_i : i \in \{ 0, \ldots, m \} \}$ be a minimal set of generators,
    so $\cornC = \her(\conv\{ X_i : i \in \{ 0, \ldots, m \} \} )$.
    By \cref{thm:finitely_generated_is_facet}, for each $X_i$,
    \begin{align}
        \mathcal{F}_i &= \{ Z \ge 0 : \Tr(X_i Z) = 1 \}
    \end{align}
    is a facet of $\dswbody^{\sharp\sharp}(S) = \dswbody(S)$.
    By \cref{thm:facet_is_clique} these facets correspond to $S$-full projectors.  Specifically, each
    $X_i$ is an $S$-full projector.
    All $S$-full projectors are in $\fp(S)$, so we have $\dswbodyS(S)$ being generated by elements
    of $\fp(S)$; therefore $\dswbodyS(S) \subseteq \fp(S)$.
    But by \cref{thm:dswbody_sandwich}, $\fp(S) \subseteq \dswbodyS(S)$.
    Therefore $\fp(S) = \dswbodyS(S)$.
\end{proof}

This invites a second possible definition of perfection for non-commutative graphs.

\begin{definition}
    A non-commutative graph $S$ is \emph{$\fg$-perfect} if $\dswbodyS(S)$ is finitely generated.
\end{definition}

\begin{theorem}
    For a graph $G$ define $S = \linspan\{ \ket{i}\bra{j} : i \simeqG j \}$
    and $S_0 = \linspan\{ \ket{i}\bra{i} : i \in V(G) \}$.
    Then $S$ is $\fg$-perfect if and only if $G$ is perfect.
\end{theorem}
\begin{proof}
    If $G$ is perfect then $\thbody^\flat(G) = \thbody(\Gc)$ is a polytope~\cite{Grtschel1986}, and
    hence is finitely generated.
    By \cref{thm:dswbodyS_classical}, $\dswbodyS(S) = \her \{ \diag(w) : w \in \thbody^\flat(G) \}$.
    If $w_1, \dots, w_n$ are generators of $\thbody^\flat(G)$ then
    $\diag(w_1), \dots, \diag(w_n)$ generate $\dswbodyS(S)$.
    So $S$ is $\fg$-perfect.

    On the other hand, suppose $S$ is $\fg$-perfect.
    By \cref{thm:dswbodyS_classical}, $\thbody^\flat(G) = \{ \diag(W) : W \in \dswbodyS(S) \}$.
    If $W_1, \dots, W_n$ generate $\dswbodyS(S)$ then
    $\diag(W_1), \dots, \diag(W_n)$ generate $\thbody^\flat(G)$.
    Being finitely generated, $\thbody^\flat(G)$ is a polytope.
    Then by~\cite{Grtschel1986} it is perfect.
\end{proof}

We close this section with a series of examples exploring the relation between
$\fp$-perfect graphs and $\fg$-perfect graphs.

\begin{example}
    It is not necessarily the case that $S$ being $\fg$-perfect implies $\Sc$ is $\fg$-perfect.
    Consider $S = \linop{\hA}$, $S_0 = \mathbb{C}I$.
    $S$ is $\fg$-perfect but $\Sc$ is not.
\end{example}
\begin{proof}
    We have $\dswbodyS(S) = \dswbodyS(\linop{\hA}) = \{ X \ge 0 : \opnorm{X} \le 1 \}$ which is generated
    by a single element: $\dswbodyS(S) = \her(\{ I \})$.
    On the other hand, $\dswbodyS(\Sc) = \{ X \ge 0 : \dim(\hA) \Tr(X) \le 1 \}$ which is not finitely
    generated: $\dim(\hA)^{-1} P$ is maximal in $\dswbodyS(\Sc)$ for any rank-1 projector $P$.
\end{proof}

\begin{example}
    $\fg$-perfect does not imply $\fp$-perfect.
    Again consider $S = \linop{\hA}$, $S_0 = \mathbb{C}I$.
    $S$ is $\fg$-perfect but is not $\fp$-perfect.
\end{example}
\begin{proof}
    That $S$ is $\fg$-perfect was shown in the previous example.
    $\fp(S) = \her(\{ I \})$ because $I$ is an $S$-full projector.
    On the other hand, there are no $S$-full projectors of $\Sc = \mathbb{C}I$,
    so $\fp(\Sc) = \{ 0 \}$ and $\fp^\abpsi(\Sc) = \psd{\hA}$.
    $\fp(S) \ne \fp^\abpsi(\Sc)$ so $S$ is not $\fp$-perfect.
\end{proof}

\begin{example}
    $\fp$-perfect does not imply $\fg$-perfect.
    A counterexample is the graph from \cref{thm:perfect_2x2}.
    This is not $\fg$-perfect, though its complement is.
\end{example}
\begin{proof}
    In \cref{thm:perfect_2x2} this graph was shown to be $\fp$-perfect, so $\dswbodyS(S) = \fp(S)$.
    By~\eqref{eq:perfect_2x2_fp}, $\fp(S)$ is not finitely generated:
    any rank-1 projector of the form
    $\left\{ \frac{1}{2} \left( \begin{smallmatrix} 1 & \phi \\ \phi^\dag & 1 \end{smallmatrix} \right)
        : \phi \in \mathbb{C}, \abs{\phi} = 1 \right\}$
    is maximal in $S$.
\end{proof}

\begin{example}
    Maximal $S$-full projectors are not necessarily vertices of $\fp(S)$, though they are maximal
    elements.
\end{example}
\begin{proof}
    Consider the graph from \cref{thm:perfect_2x2}.
    Take the maximal $S$-full projector
    $X = \frac{1}{2} \left( \begin{smallmatrix} 1 & 1 \\ 1 & 1 \end{smallmatrix} \right)$,
    corresponding to $\phi=1$ in~\eqref{eq:perfect_2x2_fp}.
    Notice that $\fp(S)$ has a tangent line at this point given by $\phi = 1 \pm i \epsilon$.
    Any supporting hyperplane of $\fp(S)$ at $X$ must contain this line.
    Therefore the intersection of all supporting hyperplanes is not dimension 0, and
    $X$ is not a vertex.
    In fact, by the symmetry of~\eqref{eq:perfect_2x2_fp}, all maximal $S$-full projectors are not
    vertices.
\end{proof}


\section{Conclusion and open questions}

Broadly speaking, there were three interrelated topics explored in this paper:
the weighed theta function, the geometry of convex corners, and perfect graphs.
While the basic theory of weighted theta functions is essentially complete,
the theory of perfect graphs stands as little more than an invitation to the topic.

There are numerous open questions surrounding all aspects of this work.
For the weighted theta function, the biggest question is whether there is any application.
In terms of what may be generalized to non-commutative graphs, the most tantalizing application
is the result of~\cite{ACIN2017489} that
independence number approaches $\lov$ when activated through the strong product:
\begin{align}
    \sup_H \frac{\alpha(G \boxtimes H)}{\lov(G \boxtimes H)} = 1.
    \label{eq:alpha_activated}
\end{align}
The first step of the proof uses $\alpha(G \boxtimes \Gc) = n$ or, rather, the weighted version of
this.  This is trivial for classical graphs since $\{(v,v) : v \in V(G)\}$ is an independent
set.  The analogous statement for non-commutative graphs seems not so easy.
Possibly the rest of the proof of~\eqref{eq:alpha_activated} could be made to work.

The entropy of the theta body is investigated in~\cite{Marton1993}.
Entropy of non-commutative convex corners is explored in~\cite[section 2.4.3]{borelandthesis}.
Is there application for the entropy of the theta body for non-commutative graphs?
Using the techniques of~\cite{Fawzi2018}, such a quantity can be estimated to arbitrary precision
via an SDP.

An alternate definition of weighted theta was presented in~\cite{btw2019}.  Does this support
a duality relation like \cref{thm:holder_dsw} or \cref{thm:DWD}?
Are maximal elements of the corresponding theta body anti-blocker in $S_0'$ like we have in
\cref{thm:dswSW}?

While we have made much progress on understanding the geometry of vertices and facets for
non-commutative convex corners, a few questions remain open.
Is it true that vertices of a convex corner always correspond to facets of the anti-blocker?
Is a convex corner defined by finitely many inequalities if and only if its anti-blocker is finitely
generated?
If a convex corner is generated by its vertices, does this mean it is finitely generated?

As for perfect graphs, we have offered nothing beyond a couple definitions and some examples.
The most pressing question is whether there is some analogue to the strong perfect graph theorem.
That is to say, can $\fp$- or $\fg$-perfection be characterized by forbidden subgraphs?
If not, is there any other characterization?

It seems unsatisfactory that $S = \mathbb{C}I$, $S_0 = \mathbb{C}I$ is not
$\fp$-perfect.
Can this be remedied by finding a tighter version of the sandwich theorem
(\cref{thm:dswbody_sandwich})?
For instance, perhaps we should replace $\fp(S)$, the non-commutative analogue of the clique
polytope, with a convex corner generated by entanglement assisted cliques (however those
may be defined).
Expanding the definition of perfect graphs by tightening the sandwich theorem has some precedent,
e.g., circular-perfect graphs~\cite{PECHER2021103224}.

Source code in the Julia language for computing weighted thetas on non-commutative graphs
is available~\cite{githubncgraphs}.


\begin{acknowledgments}
    I'd like to thank Andreas Winter and Ivan Todorov for many helpful discussions on theta
    functions and non-commutative graphs, and Eric Hanson for discussions about the
    \texttt{Convex.jl} optimization package.
\end{acknowledgments}


\appendix
\section{Proof of main theorem}
\label{sec:main_proof}

The goal of this appendix is a constructive proof of \cref{thm:thin_diag},
showing how to convert a feasible solution for $\dsw(\vecmod{S}{S_0} + \mathbb{C}I, W)$
into a feasible solution for $\dsw(S, n \blockscale(W))$, and vice versa.

\begin{lemma}
    Let $W \in \hB \ot \hZ$ be Hermitian and suppose $W$ and $\TrB W$ are invertible.  Then
    \begin{align}
        I_\hA \ot W^{-1} &\ge \proj{\Phi} \ot (\TrB W)^{-1}.
    \end{align}
    \label{lem:Winv_vs_TrWinv}
\end{lemma}
\begin{proof}
    Consider
    \begin{align}
        P &= (I_\hA \ot \sqrt{W})
        \left( \proj{\Phi} \ot (\TrB W)^{-1} \right)
        (I_\hA \ot \sqrt{W}).
    \end{align}
    This is a projector since it is Hermitian and
    (as depicted in~\cref{fig:Winv_P}),
    \begin{align}
        P^2 &=
        (I_\hA \ot \sqrt{W})
        \left( \proj{\Phi} \ot (\TrB W)^{-1} \right)
        (I_\hA \ot W)
        \left( \proj{\Phi} \ot (\TrB W)^{-1} \right)
        (I_\hA \ot \sqrt{W})
    \\ &=
        (I_\hA \ot \sqrt{W})
        (\ket{\Phi} \ot I_\hZ)
        (\TrB W)^{-1}
            (\bra{\Phi} \ot I_\hZ)
            (I_\hA \ot W)
            (\ket{\Phi} \ot I_\hZ)
        (\TrB W)^{-1}
        (\bra{\Phi} \ot I_\hZ)
        (I_\hA \ot \sqrt{W})
    \\ &=
        (I_\hA \ot \sqrt{W})
        (\ket{\Phi} \ot I_\hZ)
        (\TrB W)^{-1}
        (\TrB W)
        (\TrB W)^{-1}
        (\bra{\Phi} \ot I_\hZ)
        (I_\hA \ot \sqrt{W})
    \\ &=
        (I_\hA \ot \sqrt{W})
        (\ket{\Phi} \ot I_\hZ)
        (\TrB W)^{-1}
        (\bra{\Phi} \ot I_\hZ)
        (I_\hA \ot \sqrt{W})
    \\ &= P.
    \label{eq:Winv_P}
    \end{align}
    As a projector, we have $P \le I_\hA \ot I_\hB \ot \hZ$, so
    \begin{align}
        (I_\hA \ot \sqrt{W})
        \left( \proj{\Phi} \ot (\TrB W)^{-1} \right)
        (I_\hA \ot \sqrt{W})
        &\le I_\hA \ot I_\hB \ot I_\hZ
        \\
        \proj{\Phi} \ot (\TrB W)^{-1}
        &\le I_\hA \ot W^{-1}.
    \end{align}
\end{proof}

\begin{figure}
    \centering
    \includegraphics[scale=1.0]{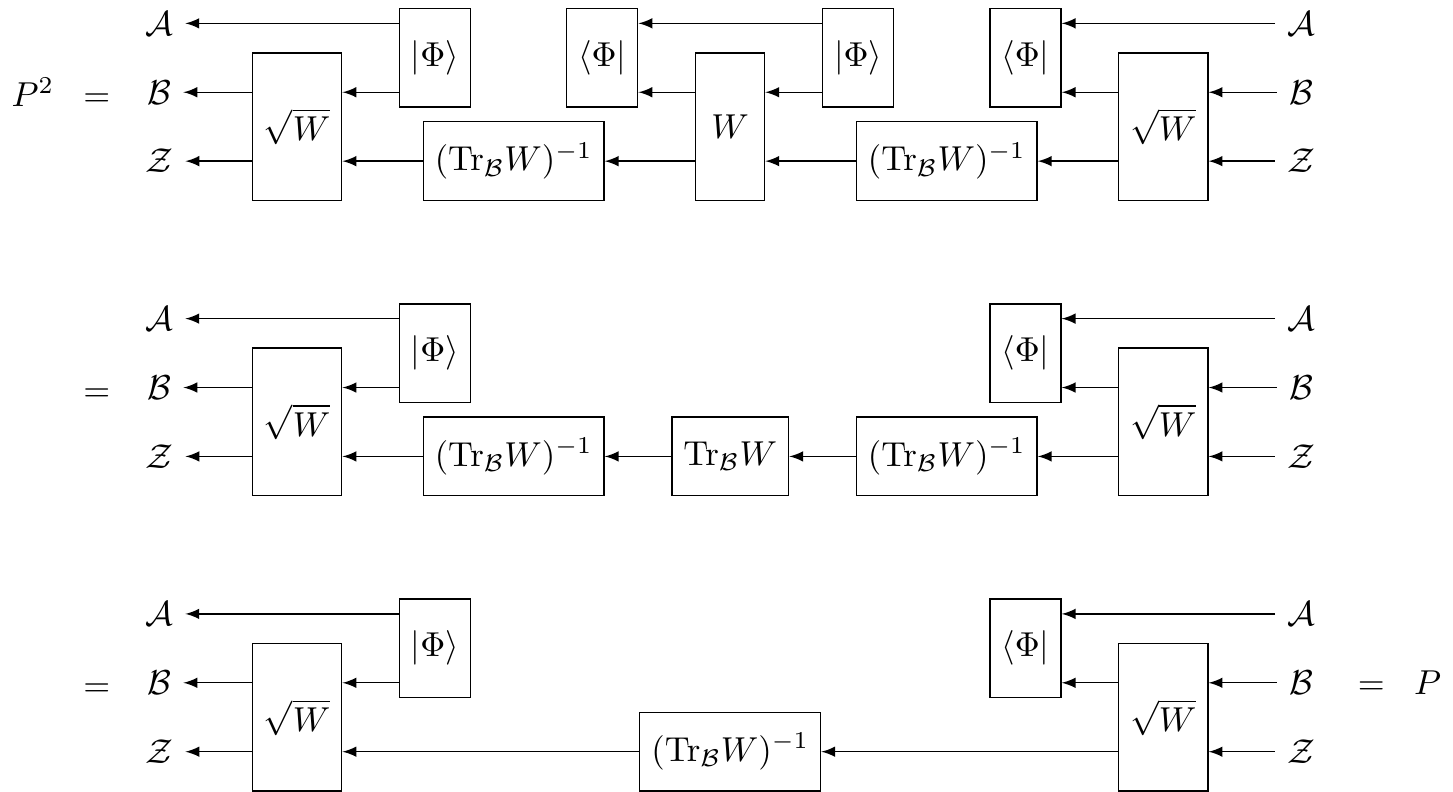}
    \caption{
        Diagram for~\eqref{eq:Winv_P}.
    }
    \label{fig:Winv_P}
\end{figure}

\begin{lemma}
    Let $G$ be a subgroup of the unitary operators on $\hA$.
    Define
    \begin{align}
        P &= \int U dU
        \\ \twirlU(\rho) &= \int U \rho U^\dag dU
    \end{align}
    where integration is with respect to the Haar measure on $G$.
    Then for $\rho \in \linop{\hA}, \rho \ge 0$,
    \begin{align}
        P \rho P^\dag \le \twirlU(\rho).
    \end{align}
    \label{thm:twirl_proj}
\end{lemma}
\begin{proof}
    We first show $P \twirlU(\rho) P^\dag = P \rho P^\dag$.
    Since we are integrating over a group,
    change of variables $U' = U W$ and $V' = W^\dag V$ gives
    \begin{align}
        P \twirlU(\rho) P^\dag &= \iiint (U W \rho W^\dag V) dU dV dW
        \\ &= \iiint (U' \rho V') dU' dV' dW
        \\ &= P \rho P^\dag.
    \end{align}
    $P$ commutes with $\twirlU(\rho)$.
    Indeed, taking $V' = UV$,
    \begin{align}
        P \twirlU(\rho)
        &= \iint (U V \rho V^\dag) dV dU
        \\ &= \iint (V' \rho V'^\dag U) dV' dU
        \\ &= \twirlU(\rho) P.
    \end{align}

    Since $P = P^\dag = P^2$, it is a projector.
    A projector that commutes with a positive semidefinite operator shrinks it,
    $P \twirlU(\rho) P^\dag \le \twirlU(\rho)$.
    Since $P \twirlU(\rho) P^\dag = P \rho P^\dag$, we are done.
\end{proof}

\begin{lemma}
    Let $G$ be a subgroup of the unitary operators on $\hA$ and define
    \begin{align}
        T &= \int (U \ot \conj{U}) dU.
    \end{align}
    If $Y \in \linop{\hA \ot \hB}, Y \ge 0$ then
    \begin{align}
        \TrA (T Y T^\dag) \le \int \left( \conj{U} (\TrA Y) \conj{U}^\dag \right) dU
    \end{align}
    where $\conj{U}$ is defined such that
    $(I \ot \conj{U}) \ket{\Phi} = (U^\dag \ot I) \ket{\Phi}$.
    \label{thm:twirl_partial_trace}
\end{lemma}
\begin{proof}
    The operators $U \ot \conj{U}$ form a group,
    $(U \ot \conj{U})(V \ot \conj{V}) = UV \ot \conj{U} \; \conj{V} = UV \ot \conj{UV}$,
    so \cref{thm:twirl_proj} applies, giving
    \begin{align}
        T Y T^\dag &\le \int \left( (U \ot \conj{U}) Y (U \ot \conj{U})^\dag \right) dU.
    \end{align}
    Taking the partial trace,
    \begin{align}
        \TrA(T Y T^\dag)
        &\le \int \TrA \left( (U \ot \conj{U}) Y (U \ot \conj{U})^\dag \right) dU.
        \\ &\le \int \TrA \left( (I_\hA \ot \conj{U}) Y (I_\hA \ot \conj{U})^\dag \right) dU.
        \\ &\le \int \left( \conj{U} (\TrA Y) \conj{U}^\dag \right) dU.
    \end{align}
\end{proof}

For a $C^*$-algebra $S_0 \subseteq{\hA}$ we have, by the structure theorem for finite dimensional
$C^*$-algebras, the decomposition~\eqref{eq:vertex_algebra}.
As explained in \cref{sec:notation}, the SDPs associated with $\dsw$ involve operators in the
augmented space $\linop{\hA \ot \hB}$, with the vector $\ket{\Phi}$ giving an isomorphism
$\hilbdual{\hA} \to \hB$.
Under this isomorphism, we can decompose the $\hB$ space in a way matching~\eqref{eq:vertex_algebra}
for $\hA$,
\begin{align}
    \hB = \bigoplus_{i=1}^r \hBi \ot \hZi.
\end{align}
This decomposition will be assumed for the remainder of this section.

\begin{lemma}
    Let $S_0 \subseteq \linop{\hA}$ be a $C^*$-algebra.
    Let $Y \in \linop{\hA \ot \hB}$ be Hermitian and $Y' = T Y T'$ where
    \begin{align}
        T = \int (U \ot \conj{U}) dU
    \end{align}
    with integration being over unitaries in $S_0$ under the Haar measure.
    Then
    \begin{align}
        Y' &= \sum_{ij}
            \ket{\Phi_{\hAi \ot \hBi}}
            \bra{\Phi_{\hAj \ot \hBj}}
            \ot Q_{ij}
        \label{eq:lem_Y_twirled}
    \end{align}
    for some $Q_{ij} \in \linop{\hYj \ot \hZj \to \hYi \ot \hZi}$.
    \label{lem:twirl_Y}
\end{lemma}
\begin{proof}
    Suppose $M \subseteq \linop{\hA}$ satisfies $U M U^\dag = M$ for all unitary $U \in S_0$.
    Then $M$ commutes with $U$.
    Also $M$ must commute with the algebra generated by these unitaries, which is all of $S_0$.
    Since the commutator of $S_0$ takes the form \eqref{eq:commutator_structure}, we have
    $M \in \bigoplus_i I_\hAi \ot \linop{\hYi}$.
    And $(M \ot I) \ket{\Phi} = \sum_i \ket{\Phi_{\hAi \ot \hBi}} \ot \ket{\psi_i}$
    with $\ket{\psi_i} \in \hYi \ot \hZi$.

    Any vector $\ket{M} \in \hA \ot \hB$ can be written as $(M \ot I) \ket{\Phi}$ for some
    $M \in \linop{\hA}$.
    If $(U \ot \conj{U}) \ket{M} = \ket{M}$ for all unitary $U \in S_0$ then $U M U^\dag = M$
    and by the above reasoning
    \begin{align}
        \ket{M} = \sum_i \ket{\Phi_{\hAi \ot \hBi}} \ot \ket{\psi_i}
        \label{eq:twirled_M_vec}
    \end{align}
    for some $\ket{\psi_i} \in \hYi \ot \hZi$.

    Because of the twirling, $Y'$ satisfies $(U \ot \conj{U}) Y' = Y'$.
    Extending~\eqref{eq:twirled_M_vec} by linearity gives
    $Y' = \sum_i \ket{\Phi_{\hAi \ot \hBi}} \ot R_i$
    with $R_i \in \hYi \ot \hZi \ot \hA^\dag \ot \hB^\dag$.
    Any Hermitian operator matching this form must take the form~\eqref{eq:lem_Y_twirled}.
\end{proof}

\begin{theorem}
    Let $S$ be an $S_0$-graph, with $S_0$ decomposed as in~\eqref{eq:vertex_algebra}.
    Let $P_i$ be the projector onto $\hAi \ot \hYi$.
    Suppose $X > 0$ commutes with all of $S_0$.
    Note that $X > 0$ requires $X$ to be invertible.
    Set $\lambda = \dsw(S, X)$.
    Then there are $Y'$ and $Q_{ij}$ such that
    \begin{align}
        Y' &\in S \ot \linop{\hB}
        \label{eq:mainlemma1_Y_in_SotB}
        \\
        Q_{ij} &\in \linop{\hYj \ot \hZj \to \hYi \ot \hZi}
        \\
        Y' &= \sum_{ij}
            \ket{\Phi_{\hAi \ot \hBi}}
            \bra{\Phi_{\hAj \ot \hBj}}
            \ot Q_{ij}
        \label{eq:mainlemma1_sumQ}
        \\
        Y' &\ge \ket{\Phi}\bra{\Phi}
        \label{eq:mainlemma1_gePhi}
        \\
        I_\hBi \ot \TrYi Q_{ii} &= \lambda (P_i X P_i)^{-T}
        \label{eq:mainlemma1_TrYiQ}
    \end{align}
    where $\ket{\Phi_{\hAi \ot \hBi}}$ is defined similar to $\ket{\Phi}$ but on the space
    $\hAi \ot \hBi$ rather than $\hA \ot \hB$.
    \label{thm:mainlemma1}
\end{theorem}
\begin{proof}
    Let $Y$ be optimal for~\eqref{eq:weighted_dsw_min_YWinvT} (of
    \cref{thm:weighted_dsw_equiv_forms})  for $\dsw(S, X)$,
    \begin{align}
        Y \in S \ot \linop{\hB},
        \\ \TrA Y \le \lambda X^{-T},
        \\ Y \ge \ket{\Phi}\bra{\Phi}.
    \end{align}
    Define the twirling operator
    \begin{align}
        T = \int (U \ot \conj{U}) dU
    \end{align}
    where we integrate over unitaries in $S_0$ under the Haar measure
    and set $Y' = T Y T^\dag$.

    Since $U S U^\dag \subseteq S_0 S S_0 \subseteq S$,
    we have
    \begin{align}
        Y' \in S \ot \linop{\hB}.
    \end{align}
    By \cref{lem:twirl_Y}, $Y'$ takes the form~\eqref{eq:mainlemma1_sumQ}.
    Since $T \ket{\Phi} = \ket{\Phi}$ we have
    \begin{align}
        Y' \ge \ket{\Phi}\bra{\Phi}.
    \end{align}
    By \cref{thm:twirl_partial_trace} we have
    \begin{align}
        \TrA Y'
        &\le \int \conj{U} (\TrA Y') \conj{U}^\dag dU
        \\ &\le \lambda \int \conj{U} X^{-T} \conj{U}^\dag dU
        \\ &\le \lambda \int X^{-T} \conj{U} \conj{U}^\dag dU
        \\ &\le \lambda X^{-T}
        \label{eq:Yprime_le_lambda_XT}
    \end{align}
    where we use the fact that $X$ (and thus $X^{-1}$) commutes with $U \in S_0$, therefore
    $X^{-T}$ commutes with $\conj{U}$.

    Considering~\eqref{eq:mainlemma1_sumQ} we have
    \begin{align}
        \TrA Y'
        &= \sum_i \TrA(P_i Y')
        \\ &= \sum_i \TrAi( \proj{\Phi_{\hAi \ot \hBi}} ) \ot \TrYi Q_{ii}
        \\ &= \sum_i I_\hBi \ot \TrYi Q_{ii}
    \end{align}
    Projecting onto $P_i^T \in \linop{\hB}$ (for a given $i$) and applying~\eqref{eq:Yprime_le_lambda_XT} gives
    \begin{align}
        I_\hBi \ot \TrYi Q_{ii} \le \lambda P_i X^{-T} P_i.
    \end{align}
    Since $P_i \in S_0$ and $X$ commutes with all of $S_0$ we have
    $P_i X^{-T} P_i = (P_i X P_i)^{-T}$ so
    \begin{align}
        I_\hBi \ot \TrYi Q_{ii} \le \lambda (P_i X P_i)^{-T}.
        \label{eq:Qii_le}
    \end{align}

    We are done aside from~\eqref{eq:Qii_le} being an inequality rather than an equality.
    This can be fixed by adding to each $Q_{ii}$ a term of the form $I_\hYi \ot \sigma_i$ with
    $\sigma_i \ge 0$.  Conditions \eqref{eq:mainlemma1_gePhi} and \eqref{eq:mainlemma1_Y_in_SotB}
    still hold, the former because we're adding a positive semidefinite term and the latter because
    $\ket{\Phi_{\hAi \ot \hBi}}
     \bra{\Phi_{\hAj \ot \hBj}}
     \ot I_\hYi \ot \sigma_i
     \in S_0 \ot \linop{\hB}
     \subseteq S \ot \linop{\hB}$.
\end{proof}

\begin{theorem}
    Let $S$ be an $S_0$-graph and take $\blockscale$ from \cref{def:blockutils}.
    Let $S' = \vecmod{S}{S_0} + \mathbb{C}I$.
    Let $W \in \psd{\hA}$ be non-singular.
    Then
    \begin{align}
        \dsw(S', W) \le \dsw(S, n \blockscale(W)).
    \end{align}
    \label{thm:mainlemma_le}
\end{theorem}
\begin{proof}
    With $S_0$ decomposed as in~\eqref{eq:vertex_algebra},
    let $P_i$ be the projector onto $\hAi \ot \hYi$.
    Set $n_i = \Tr P_i$, $n = \sum_i n_i = \dim(\hA)$,
    and $W_i = P_i W P_i$.

    Define $X = n \blockscale(W)$.
    The image of $\blockscale$ is $S_0'$ so $X$ commutes with all of $S_0$ and \cref{thm:mainlemma1} applies.
    Let $Y'$ and $Q_{ij}$ satisfy~\eqref{eq:mainlemma1_Y_in_SotB}-\eqref{eq:mainlemma1_TrYiQ}
    with $\lambda = \dsw(S, X)$.
    Substituting the definition of $X$ into~\eqref{eq:mainlemma1_TrYiQ} gives
    \begin{align}
        I_\hBi \ot \TrYi Q_{ii}
        &= \lambda (P_i X P_i)^{-T}
        \\ &= \lambda (n \dimYi^{-1} I_\hAi \ot \TrAi W_i)^{-T}
        \\ &= \lambda n^{-1} \dimYi I_\hBi \ot (\TrAi W_i)^{-T}
        \\
        \TrYi Q_{ii} &= \lambda n^{-1} \dimYi (\TrAi W_i)^{-T}.
        \label{eq:TrYiQii}
    \end{align}

    We will construct a feasible solution for $\dsw(S', W)$.
    Let $Y'_{ij} = P_i Y' P_j$.
    Because $Y'$ takes the form~\eqref{eq:mainlemma1_sumQ}, we have
    $Y'_{ij} = \ket{\Phi_{\hAi \ot \hBi}} \bra{\Phi_{\hAj \ot \hBj}} \ot Q_{ij}$.
    We will get $Y'$ into the space $S' \ot \linop{\hB}$ by adding correction terms that are
    positive semidefinite.
    We first adjust the diagonal blocks $Y'_{ii}$ to be perpendicular
    to trace-free elements of $P_i S_0 P_i$, i.e., trace free operators
    from $\linop{\hAi} \ot I_\hYi$.
    Consider the adjustment term
    \begin{align}
        Z_i &=
            \lambda n^{-1} I_\hYi \ot \left(
                I_\hAi \ot W_i^{-T}
                -
                \ket{\Phi_{\hAi \ot \hBi}}
                \bra{\Phi_{\hAj \ot \hBj}}
                \ot (\TrAi W_i)^{-T}
            \right)
    \end{align}
    which, by \cref{lem:Winv_vs_TrWinv}, is positive semidefinite (note
    $(\TrAi W_i)^{-T} = (\TrBi W_i^T)^{-1}$, and we are applying the lemma to $W_i^T$).
    Define
    \begin{align}
        Y'' = Y' + \sum_i Z_i \ge Y'.
    \end{align}
    Since $Z_i \in S_0 \ot \linop{\hB} \subseteq S \ot \linop{\hB}$, we have
    $Y'' \in S \ot \linop{\hB}$.
    For any $R \in P_i S_0 P_i$ we have, by the structure of $S_0$,
    $R = R' \ot I_\hYi$ for some $R' \in \linop{\hAi}$.
    Then
    \begin{align}
        \TrA(R Y'')
        &= \Tr_{\hAi \ot \hYi}( (R' \ot I_\hYi) (Y'_{ii} + Z_i) )
        \\ &= \Tr_{\hAi \ot \hYi}( (R' \ot I_\hYi) Y'_{ii} ) + \Tr_{\hAi \ot \hYi}( (R' \ot I_\hYi) Z_i )
        \\ &= \TrAi( R' \ket{\Phi_{\hAi \ot \hBi}} \bra{\Phi_{\hAi \ot \hBi}} ) \ot \TrYi Q_{ii}
            + \Tr_{\hAi \ot \hYi}( (R' \ot I_\hYi) Z_i )
        \\ &= R'^{T} \ot \TrYi Q_{ii}
            + \lambda n^{-1} \Tr(I_{\hYi}) \left( \Tr(R') W_i^{-T}
            - R'^T \ot (\TrAi W_i)^{-T} \right)
        \\ &= R'^{T} \ot \left( \TrYi Q_{ii} - \lambda n^{-1} \dimYi (\TrAi W_i)^{-T} \right)
            + \lambda n^{-1} \dimYi \Tr(R') W_i^{-T}.
    \end{align}
    Substituting in~\eqref{eq:TrYiQii}, the first term vanishes, leaving
    \begin{align}
        \TrA(R Y'')
        &= \lambda n^{-1} \dimYi \Tr(R') W_i^{-T}
        \\ &= \lambda n^{-1} \Tr(R) W_i^{-T}
    \end{align}
    In particular, $\TrA(P_i Y'') = \lambda n^{-1} n_i W_i^{-T}$ and for trace-free $R$,
    $\TrA(R Y'') = 0$.
    So
    \begin{align}
        Y'' \in (\vecmod{S}{S_0} + \linspan\{ P_i \}) \ot \hB.
        \label{eq:Y2_S_Pi}
    \end{align}

    Now define
    $Y''' = (I \ot \sqrt{W^T}) Y'' (I \ot \sqrt{W^T})$, giving
    \begin{align}
        \TrA Y'''_{ii} &= \lambda n^{-1} n_i \sqrt{W^T} W_i^{-T} \sqrt{W^T}.
    \end{align}
    Now, $\sqrt{W^T} W_i^{-T} \sqrt{W^T}$ is a projector because it is Hermitan and equal to its
    square,
    \begin{align}
        \left( \sqrt{W^T} W_i^{-T} \sqrt{W^T} \right)^2
        &= \sqrt{W^T} W_i^{-T} W^T W_i^{-T} \sqrt{W^T}
        \\ &= \sqrt{W^T} W_i^{-T} W_i^T W_i^{-T} \sqrt{W^T}
        \\ &= \sqrt{W^T} W_i^{-T} \sqrt{W^T}
    \end{align}
    where the second line uses that the support of $W_i^{-T}$ is $P_i$.
    Being a projector, $\sqrt{W^T} W_i^{-T} \sqrt{W^T} \le I$ and we have
    \begin{align}
        \TrA Y'''_{ii}
        &= \lambda n^{-1} n_i \sqrt{W^T} W_i^{-T} \sqrt{W^T}
        \\ &\le \lambda n^{-1} n_i I_\hB.
        \label{eq:TrA_Y3}
    \end{align}
    We can add to that to get equality.  Let $\sigma_i$ be the slack in the~\eqref{eq:TrA_Y3}
    inequality and set
    \begin{align}
        Y'''' = Y''' + \sum_i \dimAi I_\hAi \ot \sigma,
    \end{align}
    giving
    \begin{align}
        \TrA Y''''_{ii} &= \lambda n^{-1} n_i I_\hB.
    \end{align}
    We have $Y'''' \in S' \ot \linop{\hB}$.
    Indeed, since $Y'' \in (\vecmod{S}{S_0} + \linspan\{ P_i \}) \ot \hB$ by~\eqref{eq:Y2_S_Pi},
    and since we've only multiplied by $\sqrt{W^T}$ on the $\hB$ side
    and added terms in $(\vecmod{S}{S_0} + \linspan\{ P_i \}) \ot \hB$,
    we have $Y'''' \in (\vecmod{S}{S_0} + \linspan\{ P_i \}) \ot \hB$.
    But in fact $Y''''$ is perpendicular to anything in $\vecmod{\linspan\{ P_i \}}{\mathbb{C}I}$.
    For, suppose $R \in \vecmod{\linspan\{ P_i \}}{\mathbb{C}I}$.
    Then $R = \alpha_1 P_i + \dots + \alpha_n P_n$ and
    \begin{align}
        \TrA(R Y'''') &= \sum_i \alpha_i \lambda n^{-1} n_i I_\hB
        \\ &= \lambda n^{-1} I_\hB \sum_i \alpha_i n_i
        \\ &= \lambda n^{-1} I_\hB \Tr(R)
        = 0.
    \end{align}
    Therefore $Y'''' \in (\vecmod{S}{S_0} + \mathbb{C}I) \ot \linop{\hB} = S' \ot \linop{\hB}$.
    Substituting $R=I_\hA$ gives
    \begin{align}
        \TrA(Y'''') &= \lambda I_\hB.
    \end{align}
    And finally,
    \begin{align}
        Y''''
        &\ge (I \ot \sqrt{W^T}) Y' (I \ot \sqrt{W^T})
        \\ &\ge (I \ot \sqrt{W^T}) \proj{\Phi} (I \ot \sqrt{W^T})
        \\ &=\proj{\rw}.
    \end{align}
    Therefore, $Y''''$ is feasible for \cref{def:weighted_dsw_min_Y} for $\dsw(S', W)$
    with value $\lambda$.
\end{proof}

We now work toward the reverse inequality, $\dsw(S', W) \ge \dsw(S, n \blockscale(W))$.

\begin{lemma}
    Let $U_{\alpha \beta} \in \linop{\hA}$ be the generalized Pauli operators
    \begin{align}
        U_{\alpha \beta} = \sum_j \omega^{\beta j} \ket{j + \alpha} \bra{j}
    \end{align}
    where $\omega$ is a primitive root of unity of order $\dim(\hA)$
    and the indices $\alpha$ and $\beta$ range from $0$ to $\dim(\hA) - 1$.
    Then for any $M \in \linop{\hA}$ we have
    \begin{align}
        \sum_{\alpha \beta} U_{\alpha \beta} M U_{\alpha \beta}^\dag
        &= \dim(\hA) \Tr(M) I.
    \end{align}
    \label{lem:pauli_trace}
\end{lemma}
\begin{proof}
    The generalized Paulis form a basis of $\linop{\hA}$, so we can write $M$ in this basis,
    \begin{align}
        M = \sum_{\alpha' \beta'} M_{\alpha' \beta'} U_{\alpha' \beta'}.
    \end{align}
    The generalized Paulis satisfy the following braiding relation:
    \begin{align}
        U_{\alpha \beta} U_{\alpha' \beta'}
        &= \omega^{\alpha' \beta - \alpha \beta'} U_{\alpha' \beta'} U_{\alpha \beta}.
    \end{align}
    Using this braiding relation, we can evaluate the sum,
    \begin{align}
        \sum_{\alpha \beta}
        U_{\alpha \beta} M U_{\alpha \beta}^\dag
        &= \sum_{\alpha \beta \alpha' \beta'} M_{\alpha' \beta'}
            U_{\alpha \beta} U_{\alpha' \beta'} U_{\alpha \beta}^\dag
        \\ &= \sum_{\alpha \beta \alpha' \beta'} M_{\alpha' \beta'}
            \omega^{\alpha' \beta - \alpha \beta'}
            U_{\alpha' \beta'} U_{\alpha \beta} U_{\alpha \beta}^\dag
        \\ &= \sum_{\alpha \beta \alpha' \beta'} M_{\alpha' \beta'}
            \omega^{\alpha' \beta - \alpha \beta'}
            U_{\alpha' \beta'}
        \\ &= \sum_{\alpha \beta} M_{0 0} U_{0 0}
            \label{eq:pauli_twirl_00}
        \\ &= \dim(\hA)^2 M_{00} I
    \end{align}
    where~\eqref{eq:pauli_twirl_00} follows because $\sum_\alpha \omega^{\alpha \beta'} = 0$
    unless $\beta' = 0$, and similarly for the sum over $\beta$.
    Since $M_{00} = \dim(\hA)^{-1} \Tr M$, we have the desired result.
\end{proof}

The following corollary follows by linearity.

\begin{corollary}
    Let $U_{\alpha \beta} \in \linop{\hA}$ be as in \cref{lem:pauli_trace}.
    For any $M \in \linop{\hA} \ot \linop{\hB}$ we have
    \begin{align}
        \sum_{\alpha \beta} (U_{\alpha \beta} \ot I_\hB) M (U_{\alpha \beta}^\dag \ot I_\hB)
            &= \dim(\hA) I_\hA \ot \Tr_\hA M.
    \end{align}
    \label{lem:pauli_partial_trace}
\end{corollary}

\begin{theorem}
    Let $S$ be an $S_0$-graph and take $\blockscale$ from \cref{def:blockutils}.
    Let $S' = \vecmod{S}{S_0} + \mathbb{C}I$.
    For $W \in \psd{\hA}$,
    \begin{align}
        \dsw(S', W) \ge \dsw(S, n \blockscale(W)).
    \end{align}
    \label{thm:mainlemma_ge}
\end{theorem}
\begin{proof}
    With $S_0$ decomposed as in~\eqref{eq:vertex_algebra},
    let $P_i$ be the projector onto $\hAi \ot \hYi$.
    Set $n_i = \Tr P_i$, $n = \sum_i n_i = \dim(\hA)$, and $W_i = P_i W P_i$.

    Define $X = n \blockscale(W)$.
    Let $Y$ be optimal for~\eqref{eq:weighted_dsw_min_YWT_eq} for $\dsw(S', W)$,
    \begin{align}
        Y &\in S' \ot \linop{\hB},
        \\ \TrA Y &= \lambda W^T,
        \\ Y &\ge \ket{\vw}\bra{\vw},
        \label{eq:mainlemma2_thindiag_ge_Y_WW_v2}
    \end{align}
    with $\lambda = \dsw(S', W)$.

    We will construct a feasible solution for $\dsw(S, n \blockscale(W))$.
    For each $\hAi$ let
    $U_{i \alpha \beta} \in \linop{\hAi}$ be the generalized Pauli operators
    \begin{align}
        U_{i \alpha \beta} = \sum_j \omega^{\beta j} \ket{j + \alpha} \bra{j},
    \end{align}
    where $\omega$ is a primitive root of unity of order $\dimAi$.
    The indices $\alpha$ and $\beta$ range from $0$ to $\dimAi - 1$ so there are
    $\dimAi^2$ such operators.

    Define the projection and twirling operation
    \begin{align}
        K = n \sum_{i \alpha \beta} n_i^{-1} V_{i \alpha \beta}
        \ot V_{i \alpha \beta}^T.
    \end{align}
    where
    \begin{align}
        V_{i \alpha \beta} = (U_{i \alpha \beta} \ot I_\hY) P_i.
    \end{align}
    Define
    \begin{align}
        Y' &= K Y K^\dag.
    \end{align}
    Since $Y \in S' \ot \linop{\hB} \subseteq S \ot \linop{\hB}$ and
    $K \in S_0 \ot \linop{\hB}$ we have $Y' \in S_0 S S_0 \ot \linop{\hB} = S \ot \linop{\hB}$.
    Consider now its partial trace,
    \begin{align}
        \TrA Y'
        &= n^2 \sum_{i i' \alpha \alpha' \beta \beta'}
            n_i^{-1} n_{i'}^{-1}
            \Tr\left(
            (V_{i \alpha \beta} \ot V_{i \alpha \beta}^T)
            Y
            (V_{i' \alpha' \beta'}^\dag \ot V_{i' \alpha' \beta'}^{\dag T})
            \right).
        \\ &= n^2 \sum_{i i' \alpha \alpha' \beta \beta'}
            n_i^{-1} n_{i'}^{-1}
            V_{i \alpha \beta}^T
            \TrA\left( Y (V_{i' \alpha' \beta'}^\dag V_{i \alpha \beta} \ot I_\hB) \right)
            V_{i' \alpha' \beta'}^{\dag T}.
    \end{align}
    Consider the partial trace $\TrA( Y (V_{i' \alpha' \beta'}^\dag V_{i \alpha \beta} \ot I_\hB))$.
    We have $Y \in S' \ot \linop{\hB} = (\vecmod{S}{S_0} + \mathbb{C}I) \ot \linop{\hB}$
    and $V_{i' \alpha' \beta'}^\dag V_{i \alpha \beta} \ot I_\hB \in S_0 \ot \linop{\hB}$.
    These spaces only overlap on $\mathbb{C}I \ot \linop{\hB}$ so we only need to consider the
    component of $V_{i' \alpha' \beta'}^\dag V_{i \alpha \beta}$ that lies in $\mathbb{C}I$.
    The projection of this factor
    onto $\mathbb{C}I$ is
    \begin{align}
        n^{-1} I_\hA \Tr(V_{i' \alpha' \beta'}^\dag V_{i \alpha \beta})
        &= n^{-1} I_\hA \Tr(P_i)
        \delta_{i i'} \delta_{\alpha \alpha'} \delta_{\beta \beta'}
    \end{align}
    where $\delta$ is the Kronecker delta.  Continuing,
    \begin{align}
        \TrA Y' &= n \sum_{i \alpha \beta} n_i^{-2} \Tr(P_i)
            V_{i \alpha \beta}^T
            \TrA(Y)
            V_{i \alpha \beta}^{\dag T}
        \\ &= \lambda n \sum_{i \alpha \beta} n_i^{-1}
            V_{i \alpha \beta}^T
            W^T
            V_{i \alpha \beta}^{\dag T}
        \\ &= \lambda n \sum_i n_i^{-1} \sum_{\alpha \beta}
            ((U_{i \alpha \beta} \ot I_\hY) W_i
            (U_{i \alpha \beta} \ot I_\hY)^\dag)^T.
    \end{align}
    Applying \cref{lem:pauli_partial_trace},
    \begin{align}
        \TrA Y'
        &= \lambda n \sum_i n^{-1} \dimAi I_\hAi \ot \TrAi W_i
        \\ &= \lambda n \sum_i \dimYi^{-1} I_\hAi \ot \TrAi W_i
        \\ &= \lambda X,
    \end{align}
    as desired.

    Conjugating both sides of~\eqref{eq:mainlemma2_thindiag_ge_Y_WW_v2} with $K$ yields
    \begin{align}
        Y' &\ge K \ket{\vw}\bra{\vw} K^\dag.
    \end{align}
    We have
    \begin{align}
        K \ket{\vw} &=
            n \sum_{i \alpha \beta} n_i^{-1}
            (V_{i \alpha \beta} \ot V_{i \alpha \beta}^T) \ket{\vw}.
        \\ &=
            n \sum_{i \alpha \beta} n_i^{-1}
            (V_{i \alpha \beta} W \ot V_{i \alpha \beta}^T) \ket{\Phi}.
        \\ &=
            n \sum_{i \alpha \beta} n_i^{-1}
            (V_{i \alpha \beta} W V_{i \alpha \beta}^\dag \ot I_\hB) \ket{\Phi}.
        \\ &=
            n \sum_i n_i^{-1} \dimAi
            (I_\hAi \ot \TrAi W_i \ot I_\hB) \ket{\Phi}.
        \\ &=
            n \sum_i \dimYi^{-1}
            (I_\hAi \ot \TrAi W_i \ot I_\hB) \ket{\Phi}.
        \\ &= (X \ot I_\hB) \ket{\Phi}.
        \\ &= \ket{\vx}
    \end{align}
    Therefore $Y' \ge \ket{\vx}\bra{\vx}$ and $Y'$
    satisfies~\eqref{eq:weighted_dsw_min_YWT_eq} for $\dsw(S, X)$ with value $\lambda = \dsw(S', W)$.
\end{proof}


\section{Proofs for convex corner geometry}
\label{sec:corner_geometry}

\begin{proof}[Proof of \cref{thm:facet_gives_full_dim}]
    Suppose $\dim(\hA) > 1$ and that $\cornC$ has a facet.
    By \cref{thm:corn_relint} there is some $X_0$ in the relative interior of $\cornC$,
    and $P \cornC P = \cornC$ where $P$ is the projector onto the support of $X_0$.
    It must be the case that $X_0 > 0$, i.e., $X_0$ is not just in the relative interior but in
    fact in the interior.
    Otherwise the projector $P$ is not full rank, and $\cornC$ lies in a subspace
    $P \linop{\hA} P$ of dimension $\rank(X_0)^2$, strictly smaller than $\dim(\linop{\hA})-1$,
    the required dimension of a facet.
\end{proof}

\begin{proof}[Proof of \cref{thm:facet_alpha_1}]
    Suppose the supporting hyperplane defining a facet passes through the origin.
    That is, suppose the facet is of the form
    \begin{align}
        \mathcal{F} = \{ X \in \cornC : \Tr(X Y) = 0 \}.
    \end{align}
    for some $Y$, with $\Tr(X Y) \le 0$ for all $X \in \cornC$.

    By \cref{thm:facet_gives_full_dim} there is some $X_0 \in \cornC$, $X_0 > 0$.
    Then for any $Z \ge 0$ there is some $\epsilon > 0$ such that
    $0 \le \epsilon Z \le X_0$, and hence by hereditarity $\epsilon Z \in \cornC$.
    It must be the case that $Y \le 0$.  Otherwise there is some $Z \ge 0$ such that
    $\Tr(Z Y) > 0$ (e.g.\ the projector onto the positive eigenvalues of $Y$).
    Then there is some $\epsilon > 0$ such that $\epsilon Z \in \cornC$,
    and $\Tr(\epsilon Z Y) > 0$,
    contradicting that $\Tr(X Y) \le 0$ for all $X \in \cornC$.

    Let $P$ be the projector onto the support of $Y$.  The condition $\Tr(X Y) = 0$ for all
    $X \in \mathcal{F}$, combined with $Y \le 0$, requires $P X P = 0$ for all $X \in \mathcal{F}$.
    But this is inconsistent with $\mathcal{F}$ having affine dimension $\dim(\linop{\hA})-1$.
    If $P \ne 0$, $\dim(\mathcal{F}) \le (\dim(\hA) - \rank(P))^2 < \dim(\linop{\hA})-1$
    (assuming $\dim(\hA) > 1$).
    If $P = 0$ then $Y = 0$ and
    $\dim(\mathcal{F}) = \dim(\cornC) = \dim(\linop{\hA}) > \dim(\linop{\hA})-1$.
    In either case $\dim(\mathcal{F}) = \dim(\linop{\hA}) - 1$ is not possible.
    So it is not possible for the hyperplane to pass through the origin.

    A hyperplane not passing through the origin is of the form
    $\{ X \in \herm{\hA} : \Tr(X Y) = \alpha \}$ for some $\alpha \ne 0$.
    Since supporting hyperplanes satisfy $\Tr(XY) \le \alpha$ for all $X \in \cornC$,
    and since $0 \in \cornC$, it is not possible that $\alpha < 0$.
    Therefore $\alpha > 0$.
    We can then rescale, defining $Y' = Y / \alpha$ to get
    $\mathcal{F} = \{ X \in \herm{\hA} : \Tr(X Y') = 1 \}$
\end{proof}

\begin{proof}[Proof of \cref{thm:facet_is_vertex}]
    Let $X_0 \in \relint(\mathcal{F})$.
    Note first that
    \begin{align}
        X_0 + \epsilon Y \not\in \cornC
        \label{eq:XZeY_notin_C}
    \end{align}
    for all $\epsilon > 0$ because $Y$ forms a supporting hyperplane, and adding $\epsilon Y$ puts
    the point on the wrong side of that hyperplane:
    $\Tr((X_0 + \epsilon Y) Y) = 1 + \epsilon \Tr(Y^2) > 1$
    unless $Y=0$, which in turn is forbidden by $\Tr(X Y) = 1$ for $X \in \mathcal{F}$.

    We first show $Y \ge 0$.  The hypothesis $\Tr(X_0 Y) = 1$ forbids
    $Y \le 0$, so $Y$ has at least one positive eigenvalue.
    If $Y$ also has at least one negative eigenvalue then there is some
    $Z > 0$ such that $\Tr(Z Y) = 0$.
    Since $X_0 \in \relint(\mathcal{F})$ and $\Tr(Z Y) = 0$ (so $Z$ runs parallel to the hyperplane),
    $Z$ can be chosen small enough that $X_0 \pm Z \in \mathcal{F}$.
    Since $Z > 0$ there is some $\epsilon > 0$ such that $-Z \le \epsilon Y \le Z$.
    Now,
    \begin{align}
        X_0 - Z \in \mathcal{F}
        &\implies X_0 - Z \ge 0
        \nonumber
        \\ &\implies X_0 + \epsilon Y \ge 0
        \label{eq:facet_is_vert_XeY_pos}
    \end{align}
    Also, since $\epsilon Y \le Z$,
    \begin{align}
        X_0 + Z \in \mathcal{F}
        &\implies X_0 + Z \in \cornC = \cornC^{\sharp\sharp}
        \nonumber
        \\ &\implies \Tr((X_0 + Z) X) \le 1 \textrm{ for all } X \in \cornC^\sharp
        \nonumber
        \\ &\implies \Tr((X_0 + \epsilon Y) X) \le 1 \textrm{ for all } X \in \cornC^\sharp
        \label{eq:facet_is_vert_XeY_bounded}
    \end{align}
    By the second anti-blocker theorem,
    $\cornC = \cornC^{\sharp\sharp} = \{ W \ge 0 : \Tr(W X) \le 1 \textrm{ for all } X \in \cornC^\sharp \}$
    so~\eqref{eq:facet_is_vert_XeY_pos}-\eqref{eq:facet_is_vert_XeY_bounded}
    imply $X_0 + \epsilon Y \in \cornC$.
    This violates~\eqref{eq:XZeY_notin_C}.
    Therefore $Y \ge 0$.

    Since $Y \ge 0$ and $\Tr(X Y) \le 1$ for all $X \in \cornC$ (because $Y$ is a
    supporting hyperplane of $\cornC$), we have $Y \in \cornC^\sharp$.

    Every element $X \in \mathcal{F}$ defines a supporting hyperplane of
    $\cornC^\sharp$ at $Y$, because we have $\Tr(X Y) = 1$ from~\eqref{eq:polar_hyperplane}
    and $\Tr(X Y') \le 1$ for all $Y' \in \cornC^\sharp$ due to
    $X \in \cornC = \cornC^{\sharp\sharp}$.

    Since $\mathcal{F}$ has affine dimension $\dim(\linop{\hA})-1$ and that affine space doesn't
    pass through the origin (by \cref{thm:facet_alpha_1}), we have linearly independent
    $X_1, \dots, X_{\dim(\linop{\hA})} \in \mathcal{F}$.
    These are all supporting hyperplanes of $\cornC^\sharp$ at $Y$, and their intersection is the
    single point $Y$, so $Y$ is a vertex.
\end{proof}

\begin{proof}[Proof of \cref{thm:vertex_is_maximal}]
    Let $X \ne 0$ be a vertex of a convex corner $\cornC$.
    Suppose $X$ is not maximal.  Then there is $Y \ge 0, Y \ne 0$ such that $X+Y \in \cornC$.
    We will show this leads to the local shape of $\cornC$ at $X$ being flat or curved, not a sharp
    point as would be required for $X$ to be a vertex.

    Let $R_X, R_Y$ be such that $R_X R_X^\dag = X$, $R_Y R_Y^\dag = Y$, and
    $R_X R_Y^\dag + R_Y R_X^\dag \ne 0$.
    It suffices to take $R_X = \sqrt{X}$ and $R_Y = \sqrt{Y}$ unless $\Tr(X Y) = 0$, in which case
    it suffices to take $R_Y = \sqrt{Y} U$ where $U$ is a unitary causing the support of $R_Y$ to
    overlap with the support of $R_X$.

    For $0 < \epsilon < 1$ let
    \begin{align}
        Z &= X + \epsilon^2 (Y - X)
        \\ W &= \epsilon \sqrt{1 - \epsilon^2} \left( R_X R_Y^\dag + R_Y R_X^\dag \right).
    \end{align}
    Then
    \begin{align}
        Z \pm W &= (1 - \epsilon^2) X + \epsilon^2 Y \pm \epsilon \sqrt{1 - \epsilon^2}
            \left( R_X R_Y^\dag + R_Y R_X^\dag \right)
        \\ &= \left( \sqrt{1 - \epsilon^2} R_X \pm \epsilon R_Y \right)
            \left( \sqrt{1 - \epsilon^2} R_X \pm \epsilon R_Y \right)^\dag
        \\ &\ge 0
        \\
        (X+Y) - (Z \pm W) &= Y + \epsilon^2 (X - Y) \mp W
        \\ &= (1 - \epsilon^2) Y + \epsilon^2 X \mp \epsilon \sqrt{1 - \epsilon^2}
            \left( R_X R_Y^\dag + R_Y R_X^\dag \right)
        \\ &= \left( \sqrt{1 - \epsilon^2} R_Y \mp \epsilon R_X \right)
            \left( \sqrt{1 - \epsilon^2} R_Y \mp \epsilon R_X \right)^\dag
        \\ &\ge 0
    \end{align}
    so $0 \le Z \pm W \le X+Y$.  Since $X+Y \in \cornC$, by hereditarity we have
    $Z \pm W \in \cornC$.

    Let $V, \alpha$ define a supporting hyperplane of $\cornC$ at $X$ that is not parallel to $W$
    (note that the direction of $W$ does not depend on $\epsilon$),
    \begin{align}
        \Tr(V X) &= \alpha
        \\ \Tr(V T) &\le \alpha \textnormal{ for all } T \in \cornC
        \\ \Tr(V W) &\ne 0.
        \label{eq:VWne0}
    \end{align}
    If $X$ is a vertex then such a hyperplane must exist: if all hyperplanes are
    parallel to $W$ then their intersection cannot be dimension 0.
    Now consider the points $Z \pm W \in \cornC$,
    \begin{align}
        Z \pm W \in \cornC
        &\implies \Tr(V (Z \pm W)) \le \alpha
        \\ &\implies \Tr(V (Z - X \pm W)) \le 0
        \label{eq:pmVW}
    \end{align}
    As $\epsilon \to 0$, $Z-X$ scales as $O(\epsilon^2)$ whereas $W$ scales as $O(\epsilon)$.
    So for sufficiently small $\epsilon$ the $\Tr(VW)$ term dominates, requiring
    $\pm \Tr(VW) \le 0$ in contradiction of~\eqref{eq:VWne0}.
    So all supporting hyperplanes of $\cornC$ at $X$ must be parallel to $W$ and $X$ cannot be a
    vertex.
\end{proof}

\begin{proof}[Proof of \cref{thm:finitely_generated_is_facet}]
    By the symmetry of the problem we need only prove that $X_0$ is a vertex and
    forms a facet of $\cornC^\sharp$.
    Consider the semidefinite program
    \begin{align}
        p^* = \min\{ \lambda :
            &\Tr(X_0 Y) = 1,
            \nonumber
            \\ &\Tr(X_i Y) \le \lambda
                \textrm{ for } i \in \{ 1, \ldots, m \},
            \nonumber
            \\ &Y \ge 0
        \}.
        \label{eq:facet_from_vert_primal}
    \end{align}
    After showing this is feasible with $\lambda < 1$ we will show that~\eqref{eq:facet_from_vert}
    is a facet and contains $Y$.
    Intuitively, the conditions of~\eqref{eq:facet_from_vert_primal} mandate that $Y$ is a supporting
    hyperplane of $\cornC^\sharp$ at $X_0$, and remains a supporting hyperplane if ``wiggled''
    in a direction perpendicular to $X_0$.
    We will return to this at the end.

    The Lagrangian of this SDP is
    \begin{align}
        L(\lambda, Y; \mu, b_i)
        &= \lambda + \mu (1 - \Tr(X_0 Y)) + \sum_{i=1}^m b_i (\Tr(X_i Y) - \lambda)
        \\ &= \mu + \Tr\left(Y \left(\sum_{i=1}^m b_i X_i - \mu X_0 \right)\right)
            + \lambda \left(1 - \sum_{i=1}^m b_i\right).
    \end{align}
    The dual program is
    \begin{align}
        d^* = \max\Big\{ \mu :
        &\sum_{i=1}^m b_i X_i \ge \mu X_0
        \\ &\sum_{i=1}^m b_i = 1, b_i \ge 0 \Big\}.
    \end{align}
    This cannot be feasible for $\mu \ge 1$, since that would imply $X_0 \le \sum b_i X_i$, meaning
    $X_0$ is dominated by a convex combination of the other $X_i$, in contradiction to the
    generators being a minimal set.
    Therefore $d^* < 1$.
    Slater's condition is satisfied, with $Y = I / \Tr(X_0)$ and $\lambda$ large being a feasible
    point in the relative interior of~\eqref{eq:facet_from_vert_primal}.
    Therefore $p^* = d^*$, giving $p^* < 1$.

    Take $Y$ to be feasible for~\eqref{eq:facet_from_vert_primal} with $\lambda < 1$.
    Let $K_0 = 0$ and let
    $\{ K_j : j \in \{1,\dots,\dim(\linop{\hA})-1\}\}$ be a Hermitian basis of
    $\vecmod{\linop{\hA}}{X_0}$ (the space perpendicular to $X_0$ under the Hilbert-Schmidt
    inner product), with $\opnorm{K_j} \le 1$.
    With $\epsilon > 0$ small (we shall later see how small), define
    $Y_j = (1 + \epsilon \Tr(X_0))^{-1} (Y + \epsilon (I + K_j))$
    for $j \in \{0,\dots,\dim(\linop{\hA})-1\}$.
    Since $\opnorm{K_j} \le 1$, we have
    \begin{align}
        Y_j \ge 0.
        \label{eq:facet_Yj_pos}
    \end{align}
    Since each $K_j$ is orthogonal to $X_0$ and $\Tr(X_0 Y) = 1$, we have
    \begin{align}
        \Tr(X_0 Y_j)
        &= (1 + \epsilon \Tr(X_0))^{-1} \Tr(X_0 (Y + \epsilon (I + K_j)))
        \\ &= 1.
        \label{eq:facet_Yj_touch}
    \end{align}
    As for the rest of the $X_i$,
    \begin{align}
        \Tr(X_i Y_j)
        &= (1 + \epsilon \Tr(X_0))^{-1} \Tr(X_i (Y + \epsilon (I + K_j)))
        \\ &\le (1 + \epsilon \Tr(X_0))^{-1} (\lambda + \epsilon \Tr(X_i (I + K_j)))
        \\ &\le \lambda + \epsilon \Tr(X_i (I + K_j))
        \\ &< 1
    \end{align}
    where the last line uses $\lambda < 1$ and $\epsilon$ sufficiently small.

    Since $\{ X_i : i \in \{ 0, \ldots, m \} \}$ generate $\cornC$, each element of
    $\cornC$ is dominated by a convex combination of these generators.
    We have then
    \begin{align}
        \Tr(X Y_j) &\le 1
            \textrm{ for all } X \in \cornC
        \label{eq:facet_Yj_bound}
    \end{align}
    By~\eqref{eq:facet_Yj_pos} and~\eqref{eq:facet_Yj_bound}, each $Y_j \in \cornC^\sharp$.
    By~\eqref{eq:facet_Yj_touch}, the $Y_j$ are on the surface of $\cornC^\sharp$.
    Since the $K_j$ (and hence the $Y_j$) span an affine space of dimension $\dim(\linop{\hA})-1$,
    they define a facet of $\cornC^\sharp$.
    And by~\eqref{eq:facet_Yj_touch},
    $\mathcal{F}_0 = \{ Z \in \herm{\hA} : \Tr(X_0 Z) = 1 \}$ is the supporting
    hyperplane defining this facet.
\end{proof}



\end{document}